\documentclass[11pt]{amsart}
\usepackage{latexsym, amssymb, amsmath, amscd, amsfonts}
\usepackage[mathscr]{eucal}
\usepackage{mathrsfs}

\usepackage{color}
\usepackage{euler,textcomp}
\usepackage[T1]{fontenc}
\usepackage{graphicx}
\usepackage{geometry}
\usepackage{soul}
\usepackage{youngtab}

\newtheorem{theorem}{Theorem}[section]
\newtheorem*{acknowledgement}{Acknowledgement}

\newtheorem{corollary}[theorem]{Corollary}

\newtheorem{definition}[theorem]{Definition}

\newtheorem{lemma}[theorem]{Lemma}

\newtheorem{proposition}[theorem]{Proposition}
\newtheorem{example}[theorem]{Example}
\newtheorem{remark}[theorem]{Remark}

\newcommand{\M}{\mathcal{M}}

\newcommand{\OO}{\mathcal{O}}
\newcommand{\PP}{\mathcal{P}}
\newcommand{\ff}{\mathbb{F}}
\newcommand{\C}{\mathbb{C}}
\newcommand{\Z}{\mathbb{Z}}
\newcommand{\FF}{\mathcal{F}}
\newcommand{\g}{\mathfrak{g}}
\newcommand{\Sy}{\mathfrak{S}}

\input xy
\xyoption{all}

\usepackage[normalem]{ulem}

\begin{document}




\title[Categorification of Fock space]{polynomial representations and Categorifications of Fock Space}

\author{Jiuzu Hong}
\email{hjzzjh@gmail.com}
\address{School of Mathematical Sciences, Tel Aviv University, Tel Aviv
69978, Israel}
\author{Oded Yacobi}
\email{oyacobi@math.toronto.edu}
\address{Department of Mathematics, University of Toronto, Bahen Centre, 40 St. George St., Toronto, Ontario, Canada, M5S 2E4}




\begin{abstract}
The rings of symmetric polynomials form an inverse system whose  limit, the ring of symmetric functions, is the  model for the bosonic Fock space representation of the affine Lie algebra.  We categorify this construction by considering an inverse limit of categories of polynomial representation of general linear groups.  We show that this limit naturally carries an action of the affine Lie algebra (in the sense of Rouquier), thereby obtaining a famiy of categorifications of the bosonic Fock space representation.  \end{abstract}


\maketitle


\section{Introduction \label{sectionone}}
A basic object at the intersection of representation theory and algebraic combinatorics is the ring of symmetric functions in infinitely many variables.  This ring is constructed via the  limit of the inverse system
 $$
 \cdots \rightarrow B_n \rightarrow B_{n-1} \rightarrow \cdots
 $$
where $B_n=\Z[x_1,...,x_n]^{\Sy_n}$ is the symmetric polynomials in $n$ indeterminates, and $B_n\to B_{n-1}$ is the map obtained by setting $x_n=0$.  The ring $B$ of symmetric functions is then defined as the subring of $\varprojlim B_n$ consisting of elements of finite degree.

The algebra $B$ possesses a striking array of symmetries, which are related to many classical structures.  In this paper we focus on an affine Lie algebra action which realizes $B$ as the (bosonic) Fock space representation of $\widehat{\mathfrak{sl}}_n$.  Our present purpose is to categorify the limit construction of $B$ along with the action of the affine Lie algebra.

To describe this idea in more detail, fix an algebraically closed field $\mathbb{F}$ of characteristic $p
\geq 0$ and let $\g$ be the complex Kac-Moody algebra
$\hat{\mathfrak{sl}}_p$, or $\mathfrak{sl}_{\infty}$ when $p=0$.  Consider the category $\M_n$ of polynomial representations of $GL_n(\ff)$.  It is well known that $\M_n$ categorifies $B_n$, i.e. the representation ring of $\M_n$ is precisely the symmetric polynomials in $n$ variables.  Moreover, the categories $\M_n$ form a direct system
$$
 \cdots \rightarrow \M_n \rightarrow \M_{n-1} \rightarrow \cdots
 $$
 where $\M_n\to\M_{n-1}$ is the functor $V\mapsto V^{GL_1}$, i.e. the invariants with respect to the $GL_1$ in the lower right-hand corner commuting with the standard $GL_{n-1} \subset GL_n$.    (This functor categorifies the map $B_n \to B_{n-1}$.)

From this direct system we form a category $\varprojlim \M_n$, and by imposing natural finite-ness conditions we define a subcategory $\M \subset \varprojlim \M_n$  (Section 3).  $\M$ is a tensor category, whose Grothendieck group is naturally isomorphic to the ring $B$.  Our main result is that there is an action of $\g$ on $\M$ (in the sense of Rouquier) which categorifies the Fock space representation of $B$ (Theorem \ref{CatTheorem}).
This means that besides defining a family of
endofunctors on $\M$ which give an integrable representation of $\g$
on the Grothendick group $K(\M)$, we also describe the additional
data of  a degenerate affine Hecke algebra action on a certain sum
of these functors (see Section 4 for a precise definition).

As a consequence of Chuang-Rouquier theory we
obtain derived equivalences between  certain blocks of $\M$ (Corollary \ref{Derived_Equivalence}).  Using results of Brundan and Kleshchev we recover the crystal of the Fock space
naturally from our construction (Corollary \ref{Crystal}).  The vertices of the corresponding
crystal graph are the simple objects in $\M$.

The construction of a ``strong'' categorification on a  limit of categories is novel, and, we believe, interesting in its own right. Moreover it seems to be adaptable to other settings, such as representations of quantum groups. Some of the other constructions appearing in this paper have their origin in the earlier works on categorification and representation theory, such as \cite{BFK}.  Our methods are also influenced by the work of Chuang and Rouquier \cite{CR}.  They studied  $\mathfrak{sl}_2$-categorifications on $Rep(GL_n)$, the category of \emph{rational} representations of $GL_n$.  Our work can be viewed as a sort of limit of their construction, which allows us to obtain the Fock space representation, rather than exterior powers of the standard represenation.

In sequels to this work, we develop this theory from the point of view of the strict polynomial functors of MacDonald \cite{Mac} and Friedlander-Suslin \cite{FS97}, and relate these results to Khovanov's Heisenberg categorification and Schur-Weyl duality \cite{HTY},\cite{HY}. Finally, we mention that a different  $\g$-categorification of the
Fock space was  constructed in \cite{PS} using the category
$\mathcal{O}$ for the double affine Hecke algebra.  Morevoer, Stroppel and Webster recently used cyclotomic $q$-Schur algebras to categorify higher level quantum Fock space \cite{SW}.

\begin{acknowledgement}
We are grateful to Joseph Bernstein for many enlightening
discussions throughout the course of this work.  We also thank Joel
Kamnitzer, Peter Tingley, and Jinkui Wan for helpful conversations,
and Nicholas Kuhn and Antoine Touz\'e for pointing out to us the
connection of our work to the category of strict polynomial functors
of finite degree.  Finally, we thank the anonymous referee for valuable comments which greatly enhanced the clarity of this paper.  \end{acknowledgement}

\section{Preliminaries}

In this section we set some notational conventions.  Section
\ref{Sec2.1} reviews some modular representation theory of $GL_n$
and introduces the functors $R_i$ which are used frequently
throughout this paper.  In Section \ref{Fock space} we recall the
definition of the Fock space representation of $\g$.

\subsection{The general linear group}
\label{Sec2.1} Let $\ff$ be an algebraically closed field of
characteristic $p \neq 2$.  (We exclude the $p=2$ case for ease of exposition; modulo some technical complications all of our constructions carry over to that case.) Let $\ff^\infty$ be the infinite dimensional vector space with distinguished basis $\{v_1,v_2,...\}$.   We fix the standard flag $\ff^{0} \subset
\ff^{1} \subset \ff^{2} \subset \cdots$  of  vector spaces, where
$\ff^n$ is the span of $\{v_1,...,v_n\}$.

Let $G_n$ denote the rank $n$ general linear group: $$G_{n}=
GL(\ff^{n}).$$  Set $T_n \subset G_n$ to be the maximal torus
consisting of diagonal matrices,  and let $B_n$ be the Borel
subgroup of upper triangular matrices and let $B_n^-$ be the opposite Borel subgroup.

We denote by $X(T_{n})$ the character group of $T_{n}$.  We identify
$X(T_{n})$ with $n$-tuples of integers: an $n$-tuple $(k_1,...,k_n)$ gives rise to the character $diag(t_1,...,t_n)\mapsto t_1^{k_1}\cdots t_n^{k_n}$. We assume the roots appear in $B_n$ to be positive. Let $X(T_{n})_+$ be the
cone of dominant weights with respect to the positive roots;  elements of $X(T_{n})_+$ are of the form
$\lambda=(\lambda_1 \geq \cdots \geq \lambda_n) $.  For $\lambda \in
X(T_n)_+$, we set $|\lambda|=\lambda_1+ \cdots +\lambda_n$.

Let $Rep(G_n)$ denote the category of rational representations of
$G_n$. A rational representation of $G_{n}$ is said to be \emph{polynomial}
if all its matrix coefficients can be extended to polynomials on
$M_{n}$.  Denote by $\M_n$ the category of finite dimensional polynomial
representations of $G_{n}$.

Let $Z_{n} \subset G_{n}$ be the center of $G_n$, which consists of
scalar matrices.  Given a representation $V_n$ of $G_{n}$, we can
decompose it into weight spaces with respect to the action of
$Z_{n}$:
$$
V_{n} = \bigoplus_{k \in \mathbb{Z}}V_{n}(k),
$$
where $$V_{n}(k)=\{ v \in V_{n} : z.v=z^{k}v \text{ for all } z \in
Z_{n} \}.$$ The representation $V_n$ is said to be of \emph{degree}
$k$ if $V_n=V_n(k)$.  If $V_n$ is of degree $k$ for some $k$, then
we say $V_n$ is \emph{homogeneous}.  A polynomial representation of
$G_n$ is a direct sum of homogeneous representations of non-negative
degrees, and so, in particular, simple modules are homogeneous. Let
$\M_n(k)$ be the category of polynomial representations of $G_n$ of
degree $k$.

For all $n \geq 1$, we embed $G_{n-1} \subset G_{n}$ as the
automorphisms fixing $v_n$. Given $V\in Rep(G_{n})$, we denote its
restriction to $G_{n-1}$ by $V|_{G_{n-1}}$.  Let $$R:\M_n
\rightarrow \M_{n-1}$$ be the functor assigning to a representation
$V$ of $G_n$ its restriction to $G_{n-1}$: $$R(V)=V|_{G_{n-1}}.$$

Let $V$ be a representation of $G_n$.  There is an action of $G_1$
on $R(V)$, coming from the automorphisms which fix all vectors $v_i$
except $v_n$. This copy of $G_1$ commutes with $G_{n-1}\subset G_n$, i.e. $G_{n-1}\times G_1 \subset G_n$.  Since $G_1$ acts semi-simply, we can decompose the
functor
\begin{equation}
\label{Def of R_i} R=\bigoplus_{i \in \mathbb{Z}} R_{i},
\end{equation}
corresponding to the weight spaces of $G_1$.  In particular,
$R_{0}(V)=R(V)^{G_1}$, the $G_1$-invariant vectors.

  For $\lambda \in X(T_{n})$ we set
$$
H^0(\lambda)=ind_{B_n}^{G_n}(\ff_{\lambda}),
$$
where $\ff_{\lambda}$ is the one-dimensional $B_n^-$-module of weight
$\lambda$.  For
$\lambda \in X(T_{n})_+$ set
 $$
 L_n(\lambda)=socH^0(\lambda),
 $$
 the maximal semi-simple submodule of $H^0(\lambda)$.

Note that $L_n(\lambda)$ is a polynomial representation if and only if $\lambda_i\geq0$ for all $i$, where $\lambda=(\lambda_1,...,\lambda_n)$.  In general, the representation $L_n(\lambda)$ is a simple module of highest weight $\lambda$; this implies $L_n(\lambda) \in \M_n(|\lambda|)$ when $\lambda$ is a highest weight of a polynomial representation.  Moreover,
\begin{proposition}[\cite{J}, Corollay II.2.7]
\label{SimplesforGL} The representations $L_{n}(\lambda)$ with
$\lambda \in X(T_n)_+$ are a system of representatives for the
isomorphism classes of simple rational $G_n$-modules.

\end{proposition}

Recall that the Weyl group of $G_n$ is isomorphic to
$\mathfrak{S}_n$, the symmetric group on $n$ letters, which acts on
$X(T_{n})$ by place permutations.  Let $w_o$ be the longest element
of the Weyl group.

\begin{definition}
For $\lambda \in X(T_n)_+$ the \emph{Weyl module} $V_{n}(\lambda)
\in \M_n(|\lambda|)$ of highest weight $\lambda$ is given by
$$
V_{n}(\lambda)=H^0(-w_0\lambda)^*.
$$
\end{definition}

The $G_n$-module $V_{n}(\lambda)$ is generated by a $B_n$-stable
line of weight $\lambda$.  Moreover, $V_{n}(\lambda)$ is universal
with respect to this property, i.e. any $G_n$-module generated by a
$B_n$-stable line of weight $\lambda$ is a homomorphic image of
$V_{n}(\lambda)$.  In particular,
\begin{equation}
\label{Eq2} V_{n}(\lambda)/rad_GV_{n}(\lambda) \cong L_{n}(\lambda).
\end{equation}

\begin{definition}[\cite{J}, \S II.4.19]
An ascending chain $0=V_0 \subset V_1 \subset \cdots $ of submodules
of a $G_n$-module $V \in Rep(G_n)$ is called a \emph{Weyl
filtration} if $V=\bigcup V_i$ and each factor $V_i/V_{i-1}$ is
isomorphic to some Weyl module

\end{definition}

Let  $\g_{n}$ be the Lie algebra of $G_n$; this is the Lie algebra of $n\times n$ matrices over $\ff$.  The matrix units are denoted by $x_{ij}$.  Let $U(\g_{n})$ be the universal
enveloping algebra of $\g_n$ with center $Z(\g_n)$.  The Casimir operator of $U(\g_n)$ is
\begin{equation}
\label{Casimir}
C_n=\sum_{i\neq j} x_{ij}x_{ji} + \sum_{\ell=1} ^{n} x_{\ell\ell} ^2 \in Z(\g_n).
\end{equation}
For $\lambda=(\lambda_1,...,\lambda_n) \in X(T_{n})$, define the
scalar
\begin{equation}
\label{CasScalar} c_n(\lambda)=\sum_{i=1}^n
(n-2i+1)\lambda_i+\lambda_i ^2.
\end{equation}
  It is an elementary computation that on any module $V$ which is generated by a vector of weight $\lambda$, $C_n$ acts by the scalar $c_n(\lambda)$.  In particular,

\begin{lemma} \label{SCALAR}
For $\lambda \in X(T)_+$, the Casimir operator $C_n$ acts on
$V_{n}(\lambda)$ and $L_{n}(\lambda)$ by the scalar $c_n(\lambda)$.
\end{lemma}

\subsection{Kac-Moody algebras in type A}
\label{Fock space}
 Let $\mathfrak{g}$ denote the following Kac-Moody algebra (over $\C$):
\begin{equation*}
\mathfrak{g}= \left\{
\begin{array}{lr}
\mathfrak{sl}_{\infty} \text{   if } p=0  &  \\
\widehat{\mathfrak{sl}}_{p} \text{  if }  p>0
\end{array}
\right.
\end{equation*}
By definition, the Kac-Moody algebra $\mathfrak{sl}_{\infty}$ is
associated to the Dynkin diagram:
$$
\xymatrix{ \cdots \ar@{-}[r] & \bullet \ar@{-}[r] & \bullet
\ar@{-}[r] & \bullet\ar@{-}[r] & \bullet\ar@{-}[r] & \cdots}
$$
The Kac-Moody algebra $\widehat{\mathfrak{sl}}_{p}$ is associated to
the diagram with $p$ nodes: \linebreak \linebreak
$$
\xymatrix{ \bullet \ar@{-}@/^2pc/[rrrr] \ar@{-}[r] & \bullet
\ar@{-}[r] & \cdots \ar@{-}[r] & \bullet\ar@{-}[r] & \bullet }
$$
For the  precise relations defining $\g$ see \cite{Ka}. The Lie
algebra $\g$ has standard Chevalley generators $\{e_i,f_i\}_{i \in
\Z /p\Z}$. Here, and throughout, we identify $\Z/p\Z$ with the prime
subfield of $\ff$.

Set $h_i=[e_i,f_i]$.  We let $Q$ denote the root lattice and $P$ the
weight lattice of $\g$.  The cone of dominant weights is denoted
$P_+$ with generators the fundamental weights $\{\varpi_i:i \in
\Z/p\Z\}$.  Let $\mathcal{L}(\omega)$ be the irreducible $\g$-module
of highest weight $\omega \in P_+$.  In particular, the \emph{basic
representation} is $\mathcal{L}(\varpi_0)$.

Of central importance to us is the \emph{Fock space} representation
of $\g$, which we now define.  Let $\Lambda$ be the set of all
partitions.   As a vector space $\mathcal{F}$ has basis indexed by
partitions:
$$
\mathcal{F}=\bigoplus_{\lambda \in \Lambda}\C v_{\lambda}.
$$
We identify partitions with their Young diagram (using English
notation).  For example the partition $(4,4,2,1)$ corresponds to the
diagram
 $$
 \yng(4,4,2,1)
 $$
The \emph{content} of a box in position $(k,l)$ is the integer $l-k
\in \Z/p\Z$.  Given $\mu,\lambda \in \Lambda$, we write
\begin{equation*}
\xymatrix{ \mu \ar[r] & \lambda}
\end{equation*}
if $\lambda$ can be obtained from $\mu$ by adding some box.  If the
arrow is labelled $i$ then $\lambda$ is obtained from $\mu$ by
adding a box of content $i$ (an $i$-box, for short).  For instance,
if $\mu=(2)$, $\lambda=(2,1)$, and $m=3$ then
$$
\xymatrix{
 \mu \ar[r]^{2} & \lambda}.
$$
An $i$-box of $\lambda$ is \emph{addable} (resp. \emph{removable})
if it can be added to (resp. removed from) $\lambda$ to obtain
another partition. Let $m_i(\lambda)$ denote the number of $i$-boxes
of $\lambda$. Define
\begin{equation}
\label{milambda}
n_i(\lambda)=m_{i-1}(\lambda)+m_{i+1}(\lambda)-2m_{i}(\lambda)+\delta_{i0},
\end{equation}
where $\delta_{i0}$ is the Kronecker delta.


The action of $\g$ on $\mathcal{F}$ is given on Chevalley generators
by the following formulas:
$$
e_i.v_{\lambda}=\sum v_{\mu}
$$
the sum over all $\mu$ such that $\xymatrix{ \mu \ar[r]^{i} &
\lambda}$, and
 $$
f_i.v_{\lambda}=\sum v_{\mu}
$$
the sum over all $\mu$ such that $\xymatrix{ \lambda \ar[r]^{i} &
\mu}$.   These equations define an integral  representation of
$\mathfrak{g}$ (see e.g. \cite{LLT}).

Note that $v_{\emptyset}$ is a highest weight vector of highest
weight $\omega_0$. We note also that the standard basis of $\FF$ is
a weight basis,  any weight vector $v_\lambda$ is of weight
\begin{equation}
\label{weight} {\omega_0-\sum_i m_i \alpha_i.}
\end{equation}
Indeed, from the above formulas one obtains:
$$h_i.v_{\lambda}=n_i(\lambda)v_{\lambda}.$$


\section{The category $\M$ }

In this section we define our main object of study, the category
$\M$.

\subsection{The stable category}\label{Sec3.2}
In this section we describe a stable category constructed from
polynomial representations of general linear groups. Recall the
functor $R_0:\M_{n+1} \rightarrow \M_{n}$, which is defined by
taking the invariants with respect to $G_1$:
$$R_0(V_{n+1})=V_{n+1}^{G_1}.$$ Since $R_0$ preserves degree, it
also defines a functor $R_0:\M_{n+1}(k) \rightarrow \M_n(k)$.

\begin{proposition}[Theorem 4.3.6, \cite{M93}]
\label{R_0Equivalence} For $n \geq k$, $R_0:\M_{n+1}(k) \rightarrow
\M_n(k)$ is an equivalence of categories.
\end{proposition}

We have a diagram of functors:
$$
\cdots \overset{R_{0}}{\longrightarrow }
 \M_2 \overset{R_{0}}{\longrightarrow }
 \M_1 \overset{R_{0}}{\longrightarrow }
 \M_0.
$$
The category $\M$ will be constructed by first defining the limit
category of this diagram, and then taking the subcategory of compact
objects.

\begin{definition}
Let $\widetilde{\M}$ be the category whose objects are sequences
$$V=(V_{n},\alpha_{n})_{n=0}^{\infty}$$ where $V_{n} \in \M_n$ and
$$\alpha_{n}:R_{0}(V_{n+1}) \rightarrow V_{n}$$ is an isomorphism of
$G_{n}$-modules (by convention $G_0$ is the trivial group).

For objects  $ V=(V_{n},\alpha_{n})_{n=0}^{\infty}$ and
$W=(W_{n},\beta_{n})_{n=0}^{\infty}$ in $\widetilde{\M}$, the space
of morphisms $Hom_{\widetilde{\M}}(V,W)$ is defined as the limit
$$\underleftarrow{lim}Hom_{G_n}(V_{n},W_{n}).$$  Here we are using the maps
$$
\theta^{n+1}_{n}:Hom_{G_{n+1}}(V_{n+1},W_{n+1}) \rightarrow
Hom_{G_n}(V_{n},W_{n})
$$
given by $f_{n+1} \mapsto \beta_{n}\circ R_{0}(f_{n+1}) \circ
\alpha_{n}^{-1}$.

\end{definition}

By definition, a morphism $f \in Hom_{\widetilde{\M}}(V,W)$ is a
sequence $(f_n)_{n=0}^{\infty}$, where $$f_{n} \in
Hom_{G_n}(V_{n},W_{n})$$ and the following diagram commutes
$$
\xymatrix{
R_{0}(V_{n+1}) \ar[r]^{R_{0}(f_{n+1})} \ar[d]^{\alpha_{n}} & R_{0}(W_{n+1}) \ar[d]^{\beta_{n}} \\
V_{n} \ar[r]^{f_{n}} & W_{n} }
$$
for $n \gg 0$.  Two morphisms $f,g:V \rightarrow W$ are equal if for
$n \gg 0$, $f_{n}=g_{n}$.

When there is no cause for confusion, we abbreviate an object
$(V_n,\alpha_n)_{n=0}^{\infty}$ in $\widetilde{\M}$ by
$(V_n,\alpha_n)$, or sometimes simply $(V_n)$ when the maps
$\alpha_n$ are understood from context.  Similarly a morphism
$(f_n)_{n=0}^{\infty}$ will be abbreviated as $(f_n)$.

If $V_n,W_n \in \M_n$, then canonically $R_0(V_n \otimes W_n) \cong
R_0(V_n)\otimes R_0(W_n)$.  Therefore there is a monoidal structure
on $\widetilde{\M}$ described as follows: for $V=(V_n,\alpha_n)$ and
$W=(W_n,\beta_n)$ objects in $\widetilde{\M}$, set $V \otimes W
=(V_n \otimes W_n, \alpha_n \otimes \beta_n)$. The following lemma
is straight-forward.
\begin{lemma}
$\widetilde{\M}$ is a symmetric tensor category.
\end{lemma}

\begin{example}
\label{Ex of Objects} Here are some naturally occurring objects in
$\widetilde{\M}$.
\begin{enumerate}
\item Let $1_n$ be the trivial one dimensional representation of $G_n$.  Clearly, $R_0(1_{n+1}) \cong 1_n$, so the trivial representations glue together to the unit object $1=(1_n)$ in the symmetric tensor category $\widetilde{\M}$.

\item Consider the standard representation of $G_n$ on $\ff^n$.  Then $R_0(\ff^{n+1})$ is precisely the span of $v_1,...,v_n$, so using the obvious morphisms $\iota_n:R_0(\ff^{n+1}) \rightarrow \ff^n$, we obtain the ``standard'' object $St=(\ff^n, \iota_n) \in  \widetilde{\M}$.

\item Fix a nonnegative integer $r$, and consider the tensor product representation of $G_n$ on $\bigotimes^r \ff^n$.  Then, as above, there are obvious morphisms $R_0(\bigotimes^r \ff^{n+1}) \rightarrow \bigotimes^r \ff^n$, which glue together to an object which is canonically isomorphic to $\bigotimes^r St$.  One can similarly define objects $S^r(St)$ and $\bigwedge^rSt$.

\item In contrast to the tensor algebra and symmetric algebra of $\ff^n$, the exterior algebra is finite dimensional so $\bigwedge\ff^n \in \M_n$. Therefore, from isomorphisms $R_0(\bigwedge \ff^{n+1})
\rightarrow \bigwedge \ff^n$, we can define an object which we
denote $\bigwedge St =(\bigwedge \ff^n) \in \widetilde{\M}$.
\end{enumerate}
\end{example}

\begin{definition}
An object $V=(V_{n}) \in \widetilde{\M}$ is of \emph{degree} $k$ if
for every $n$, $V_{n}$ is of degree $k$.  If $V$ is of degree $k$
for some $k$, then we say $V$ is \emph{homogeneous}.  Let $\M(k)$
denote the subcategory of $\widetilde{\M}$ consisting of objects of
degree $k$.
\end{definition}

For an object $V=(V_n) \in \widetilde{\M}$, let $V(k)=(V_{n}(k))$.
We have the (possibly infinite) direct sum in $\widetilde{\M}$:
 \begin{equation}
 \label{compact}
 V=\bigoplus V(k).
 \end{equation}
 An object $V \in \widetilde{\M}$ is \emph{compact} if the direct sum above is finite.  This is equivalent to the usual notion of compactness from category theory, namely that $V$ commutes with coproducts.  For instance, in the example above (1)-(3) are compact objects, while $\bigwedge St$ is not.  We thus arrive at a different realization of the category of polynomial functors.

\begin{definition}
\label{DefCatP} The category $\M$ is the full subcategory of
$\widetilde{\M}$ consisting of compact objects, i.e. those objects
where the sum (\ref{compact}) is finite.
\end{definition}

Note that $\M$ is the direct sum of categories:
$$
\M=\bigoplus_{k=0}^{\infty}\M( k),
$$
and that $\M$ is also a tensor category. Let $\Psi_{n}$ be the
projection functor from $\M$ to $\M_n$, and let $\Psi_n(k)$ denote
its restriction to $\M(k)$. By Proposition \ref{R_0Equivalence},
we have:
\begin{proposition}
\label{Forgetful-Functor} For $n\geq k$, $\Psi_{n}(k):
\M(k)\rightarrow \M_n(k)$ is an equivalence of categories.
\end{proposition}

\section{Properties of $\M$}
In this section we undertake a thorough study of the category $\M$
in preparation for the categorification theorem we prove in the next
section. In Section \ref{SectionPind} we use polynomial induction
functors (\cite{BK2}) to construct an inverse functor to $R_0$. In
Section \ref{Sec3.3} we study some standard objects in $\M$.  In
Section \ref{SectionFunctors} we introduce the functors $E$ and $F$
and decompose them into subfunctors using endomorphisms $X \in
End(E)$ and $Y \in End(F)$.

\subsection{Polynomial induction}\label{SectionPind}
We recall the notion of polynomial induction due to Brundan and
Kleshchev \cite{BK2}, and give self-contained proofs of a few of
their results.

Set $M_{m,n}=Hom(\ff^n,\ff^m)$, and for convenience let
$M_n=M_{n,n}$.  Let $\OO(M_{m,n})$ be the algebra of polynomials on
$M_{m,n}$.  There exists a natural action of $G_m \times G_n$ on
$M_{m,n}$ by
$$(g_1,g_2)\cdot A=g_1Ag_2^t.$$      Then $G_m \times G_n$ acts on
$\OO(M_{m,n})$ by
$$((g_1,g_2) \cdot f)(A)=f(g_1^tAg_2).$$

\begin{definition}
\label{FunctorP_mn} Let $I_{0}$ be the functor from $Rep(G_n)$ to
$Rep(G_{n+1})$ given by
$$
V_n \mapsto (V_n\otimes\OO(M_{n,n+1}))^{G_n}.
$$
\end{definition}
Here the invariants are taken with respect to the tensor product
action of $G_n$ on $V_n\otimes\OO(M_{n,n+1})$.  This action commutes
with the action of $G_{n+1}$ on $\OO(M_{n,n+1})$ by right
translation, thus resulting in a $G_{n+1}$-module. By Proposition
A.3 in \cite{J},  $I_0(V_n)$ is a polynomial representation.

Now let $O_P:Rep(G_n) \rightarrow Rep(G_n)$ be the functor assigning to
a representation its maximal polynomial submodule.   Define
$Ind_{n}^{n+1}:Rep(G_n) \rightarrow Rep(G_{n+1})$ by
$$
Ind_{n}^{n+1}(V_n)=(V_n\otimes \mathcal{O}(G_{n+1}))^{G_n},$$ where
$\mathcal{O}(G_{n+1})$ denotes the algebra of regular functions on
$G_{n+1}$. Here, $G_{n+1}$ acts on $\mathcal{O}(G_{n+1})$ by right
translation, while $G_n$ acts by left translation.
\begin{definition}
\label{DefPind} The \emph{polynomial induction} functor, denoted $Pind_n ^{n+1}$, is the composition of
functors $O_P\circ Ind_{n}^{n+1}$.
\end{definition}

By Equation (1) in Section A.18 in \cite{J},
$O_P(\mathcal{O}(G_{n+1}))=\OO(M_{n+1})$.  From this we obtain
another formulation of polynomial induction: $$Pind_n ^{n+1}(V_n)
\simeq(V_n \otimes \OO(M_{n+1}))^{G_n}.$$

\begin{lemma}
\label{Pind decomp} There exists a natural $G_{n+1}$-isomorphism
$$Pind_n^{n+1}(V_n) \simeq S(\ff^{n+1})\otimes I_{0}(V_n).$$  Here
$S(\ff^{n+1})$ denotes the symmetric algebra of $\ff^{n+1}$, endowed
with the standard $G_{n+1}$ action induced from the natural
representation, and $G_{n+1}$ acts on $S(\ff^{n+1})\otimes
I_{0}(V_n)$ by the tensor product action.

\end{lemma}
\begin{proof}
As a $G_n\times G_{n+1}$-module $M_{n+1}^* \cong \ff^{n+1}\oplus
M_{n,n+1}^*$.  Here $M_{n+1}^*$ and $M_{n,n+1}^*$ are the linear duals of $M_{n+1}$, and $M_{n,n+1}$ and  $ \ff^{n+1}$ is
considered a $G_n \times G_{n+1}$-module with $G_n$ acting
trivially.  Hence as a $G_n\times G_{n+1}$-module $\OO(M_{n+1})
\cong S(\ff^{n+1})\otimes \OO(M_{n,n+1})$, where $S(\ff^{n+1})$ is
trivial as a $G_n$-module.  The result follows.
\end{proof}

\begin{lemma}
\label{Lemma R_0 adj P} $I_0$ defines a functor $\M_n\to\M_{n+1}$, and $R_0:\M_{n+1} \rightarrow
\M_{n}$ is left adjoint to $I_{0}:\M_n \rightarrow \M_{n+1}$.
\end{lemma}
\begin{proof}
It suffices to show that for all $V_{n+1} \in \M_{n+1}$ and $W_n \in
\M_n$
$$Hom_{G_n}(R_0(V_{n+1}),W_n)\cong Hom_{G_{n+1}}(V_{n+1},I_{0}(W_n)).
$$
Since the functors $R_0$ and $I_0$ preserve degree, we can assume
that $V_{n+1}$ and $W_n$ are of the same degree $k$.  In this case
we have
\begin{eqnarray*}
Hom_{G_{n}}(R_0(V_{n+1}),W_n) & \cong & Hom_{G_{n}}(R(V_{n+1}),W_n)\\
&\cong& Hom_{G_{n+1}}(V_{n+1},Pind_n^{n+1}W_n)
\\ &\cong& Hom_{G_{n+1}}(V_{n+1},I_{0}(W_n)).
\end{eqnarray*}
Note that the second isomorphism follows from Frobenius
reciprocity, and the last one follows from Lemma \ref{Pind decomp} by degree considerations.
\end{proof}

Let $I_{n,n+1}$ be the $n\times n+1$ matrix with $1$ in entry $(i,i)$ for $1\leq i \leq n$ and $0$ elsewhere.

\begin{corollary}
\label{poly ind inverse} For $n \geq k$,
$$I_{0}:\M_n(k)\rightarrow \M_{n+1}(k)$$ is an inverse functor, up to isomorphism,
of $R_0$.  Moreover, the isomorphism $R_0I_0 \rightarrow Id$ is
given by evaluation at $I_{n,n+1}$.
\end{corollary}

\begin{proof}
Recall that an adjoint of an equivalence is necessarily isomorphic to an inverse functor.  Therefore, by Proposition \ref{R_0Equivalence} and Lemma \ref{Lemma R_0 adj P}, the first statement of the corollary follows.  This implies that the counit $R_0I_0\to Id$ is an isomorphism.  It remains to show that it is given by evaluation at $I_{n,n+1}$.

The counit evaluated at $V_n\in\M_n$ is the morphism $R_0I_0(V_n)\to V_n$ obtained by passing the identity morphism of $I_0(V_n)$ through the chain of isomorphisms appearing in the proof of Lemma \ref{Lemma R_0 adj P}.  Viewing elements of $Pind_n^{n+1}(V_n)$ as functions from $M_{n+1}$ to $V_n$, and elements of $I_0(V_n)$ as functions from $M_{n,n+1}$ to $V_n$, the corollary follows from the following two observations.  Firstly, the isomorphism $Hom_{G_{n+1}}(V_{n+1},I_{0}(W_n))\cong Hom_{G_{n+1}}(V_{n+1},Pind_n^{n+1}W_n)$ maps a function $f:V_{n+1}\to I_0(W_n)$ to $(v\mapsto f(v)\circ pr)$, where $pr$ is the projection from $M_{n+1} \to M_{n,n+1}$.  Secondly, the Frobenius reciprocity isomorphism $Hom_{G_{n+1}}(V_{n+1},Pind_n^{n+1}W_n)\cong Hom_{G_{n}}(R(V_{n+1}),W_n)$ maps $f:V_{n+1}\to Pind_n^{n+1}W_n$ to $(v \mapsto f(v)(I_{n+1}))$.
\end{proof}

\begin{lemma}\label{Weyl-Module-Ind}
Let $n \geq k$ and $\lambda \in X(T_n)_+$ such that $|\lambda|=k$.
Then $R_0(V_{n+1}(\lambda)) \simeq V_n(\lambda)$ and
$I_{0}(V_n(\lambda)) \simeq V_{n+1}(\lambda)$.
\end{lemma}

\begin{proof}
The first claim follows by the following series of equalities of
characters:
 \begin{eqnarray*}
 char_{\ff}(R_0(V_{n+1}(\lambda))) &=& char_{\C}(R_0(V_{n+1}(\lambda))) \\ &=& char_{\C}(V_{n}(\lambda)) \\ &=& char_{\ff}(V_n(\lambda)).
 \end{eqnarray*}
The first equality is a consequence of that fact that
$R_0(V_{n+1}(\lambda))$ admits a Weyl filtration (Theorem
\ref{Donkin}).  The second equality follows by the classical
branching rules for the general linear groups, and the last one by
definition of Weyl modules.

The second isomorphism in the statement of the lemma follows from
the first by Corollary \ref{poly ind inverse}.\qedhere

\end{proof}

Given $(V_n,\alpha_n) \in \M$, let
$\alpha_n^{\vee}:V_{n+1}\rightarrow I_0(V_n)$ be the morphism
induced from $\alpha_n:R_0(V_{n+1})\rightarrow V_n$ by Lemma
\ref{Lemma R_0 adj P}.  Then   $\alpha_n^{\vee}$ is an isomorphism,
and Lemma \ref{Lemma R_0 adj P} and Corollary \ref{poly ind inverse} imply
the following:
\begin{lemma}
\label{Diagram10} We have the following commutative diagram:
$$
\xymatrix{ R_0(V_{n+1})\ar[r]^<<<<<{\alpha_n^{\vee}}
\ar[dr]^{\alpha_n} &R_0 \circ I_0(V_n) \ar[d]^{ev_{I_{n,n+1}}} \\
& V_n }
$$
\end{lemma}

\subsection{Objects in $\M$}  We parameterize the Weyl (i.e. standard) and simple objects in $\M$.
\label{Sec3.3}

\subsubsection{The projection functor}

By Proposition \ref{Forgetful-Functor} we have the following
equivalence of categories:
\begin{equation}
\label{BigEquivalence} \Psi:\bigoplus_{k=0}^{\infty} \M(k)
\rightarrow \bigoplus_{k=0}^{\infty} \M_k(k)
\end{equation}
which is defined by taking the direct sum of projection functors
$\Psi_{k}(k)$.  We now describe the inverse functor to $\Psi$.
\begin{definition}
To a representation $V_k \in \M_k(k)$, we associate an object
$\Psi^{-1}(V_k) \in \M$ as follows:
$\Psi^{-1}(V_k)=(V_n)_{n=0}^{\infty}$ where for $n \geq k$
$$
V_n=I_0^{n-k}(V_k)
$$
and for $n<k$
$$
V_n=R_0^{k-n}(V_k).
$$

\end{definition}
Since $I_{0}$ and $R_0$ are inverse functors to each other for $n
\geq k$ (Corollary \ref{poly ind inverse}), this gives a
well-defined object in $\M$.  The same formulas apply to morphisms
as well; we thus obtain a functor $\Psi^{-1}:\M_k \rightarrow \M$.
By taking direct sum we obtain the functor
$$
\Psi^{-1}:\bigoplus_{k=0}^{\infty} \M_k(k) \rightarrow \M.
$$
The next proposition, which follows directly from the fact that
$I_{0}$ and $R_0$ are inverse functors up to isomorphism,  justifies
our choice of notation.
\begin{proposition}
\label{InvertoS} The functor $\Psi^{-1}$ is inverse to $\Psi$.
\end{proposition}

\subsubsection{Simple objects in $\M$}
Recall that $\Lambda$ is the set of all partitions.  Let $\Lambda_k$
be the set of partitions of $k$.

\begin{proposition}
The simple objects of $\M$ are in canonical one-to-one
correspondence with $\Lambda$.
\end{proposition}

\begin{proof}

By Proposition \ref{SimplesforGL}, the simple objects of $\M_k(k)$
are, up to isomorphism, precisely $$\{L_k(\lambda): \lambda \in
\Lambda_k \}$$ where we identify the partition $\lambda \in
\Lambda_k$ with the weight $\lambda \in X(T_k)_+$.  It follows that
the simple objects on the right hand side of (\ref{BigEquivalence})
are, up to isomorphism, given by
$$
\{ L_{|\lambda|}(\lambda): \lambda \in \Lambda\}.
$$
Since $\Psi$ is an equivalence the result follows.
\end{proof}

Let $L(\lambda) \in \M$ denote the simple object corresponding to
$L_{|\lambda|}(\lambda)$ under $\Psi$, i.e.
$$L(\lambda)=\Psi^{-1}(L_{|\lambda|}(\lambda)).$$ We take now the
opportunity to record that Schur's lemma holds in $\M$.

\begin{lemma}
\label{Schurslemma} Let $L,L' \in \M$ be simple objects. Then
$$
Hom_{\PP}(L,L') \simeq \left\{
\begin{array}{lr}
\ff \text{   if } L \simeq L'  &  \\
0 \text{  otherwise }
\end{array}
\right.
$$
\end{lemma}

\subsubsection{Weyl objects in $\M$}

\begin{definition}
\begin{enumerate}
\item An object $V=(V_n)\in \M$ is called a \emph{Weyl object}, if for $n \gg0$,
$V_n$ is isomorphic to some Weyl module of $G_n$.
\item An ascending chain $0=V^0\subset V^1\subset \cdots \subset V^k=V$ of sub-objects of $V \in \M$ is called a \emph{Weyl filtration} if each factor $V^i/V^{i-1}$
is isomorphic to some Weyl object.
\end{enumerate}
\end{definition}



\begin{proposition}
\label{WeylObjects} For $\lambda \in \Lambda_k$ let
$V(\lambda)=\Psi^{-1}(V_k(\lambda)) \in \M$. Then $V(\lambda)$ is a
Weyl object in $\M$, and the Weyl objects of $\M$ are up to
isomorphism precisely the collection
$$
\{ V(\lambda): \lambda \in \Lambda\}.
$$
\end{proposition}

\begin{proof}

By Lemma \ref{Weyl-Module-Ind}, $V(\lambda)$ is a Weyl object.
Moreover, it is clear that $V(\lambda) \ncong V(\mu)$, for $\lambda
\neq \mu$.  Finally, given a Weyl object $V$ then for some $k \geq
0$, $\Psi_{k}(V) \simeq V_k(\lambda)$, where $\lambda \in
\Lambda_k$.  Therefore by Proposition \ref{InvertoS}, $V \simeq
V(\lambda)$.
\end{proof}

\subsection{Functors on $\M$}
Our aim now is to define the endofunctors $E_i$ and $F_i$ on $\M$
that will be shown in the next section to categorify the action of
the Chevalley generators on Fock space. \label{SectionFunctors}
\subsubsection{The functors $E$ and $F$}

\begin{definition}
$F:\M \rightarrow \M$ is given by $F(V)=St \otimes V$.
\end{definition}
By definition of the tensor structure on $\M$, on morphisms $F$ is
given by $F(f)=(\iota_n \otimes f_n)$, where $f=(f_n)$ (see Example
\ref{Ex of Objects} for the definition of $\iota_n$).

To define $E:\M \rightarrow \M$ consider first the element $s_n \in
G_{n+1}$ given by,
\begin{equation}
\label{matrix} \left(
  \begin{array}{ccccc}
    1 &  &  &  &  \\
     & \ddots &  &  &  \\
     &  & 1 &  &  \\
     &  &  & 0 & -1 \\
     &  &
      & 1 &  0\\
  \end{array}
\right)
\end{equation}

\begin{lemma}{\label{Transposition}}
Let $n \geq 1$.  Consider the functors $R_1 \circ R_0$ and $R_0
\circ R_1$ from $\M_{n+1}$ to $\M_{n-1}$.  The action of $s_n$
induces a natural isomorphism $$s_n: R_1 \circ R_0 \rightarrow R_0
\circ R_1.$$
\end{lemma}
\begin{proof}
  This lemma follows from the fact that $s_n$ commutes with $G_{n-1}$, and an elementary computation:
$$
\left(
  \begin{array}{cc}
    0 & -1 \\
    1 & 0 \\
  \end{array}
\right) \left(
  \begin{array}{cc}
    t_1 & 0 \\
    0 & t_2 \\
  \end{array}
\right) \left(
  \begin{array}{cc}
    0 & 1 \\
    -1 & 0 \\
  \end{array}
\right) = \left(
  \begin{array}{cc}
    t_2 & 0 \\
    0 & t_1 \\
  \end{array}
\right).
$$
\qedhere
\end{proof}

\begin{definition}
Define a functor $E:\M \rightarrow \M$ as follows.  Let
$V=(V_n,\alpha_n) \in \M$.  Then $E(V)=(E(V)_n,\tilde{\alpha}_n),$
where $$E(V)_n=R_1(V_{n+1})$$ and
$$\tilde{\alpha}_n=R_1(\alpha_{n+1}) \circ s_{n+1}^{-1}.$$Given a
morphism $$f=(f_n):(V_n,\alpha_n) \rightarrow (W_n,\beta_n),$$ set
$E(f)=(E(f)_n)$, where
$$
E(f)_n=R_1(f_{n+1}).
$$
\end{definition}
\begin{lemma}
The functor $E:\M \rightarrow \M$ is well-defined.
\end{lemma}

\begin{proof}
Let $V \in \M$.  It is clear that $E(V)\in\M$.  Suppose
$$f=(f_n):V=(V_n,\alpha_n) \rightarrow W=(W_n,\beta_n)$$ is a
morphism. We need to check that $E(f) \in Hom_{\PP}(E(V),E(W))$.

Consider the following diagram,
$$
\xymatrix{ R_1(R_0(V_{n+1})) \ar[r]^{s_n} \ar[d]^{R_1(R_0(f_{n+1}))}
& R_0(R_1(V_{n+1})) \ar[r]^>>>>>>{\tilde{\alpha}_{n-1}}
\ar[d]^{R_0(R_1(f_{n+1}))} & R_1(V_n) \ar[d]^{R_1(f_{n})} \\
R_1(R_0(W_{n+1})) \ar[r]^{s_n} & R_0(R_1(W_{n+1}))
\ar[r]^>>>>>{\tilde{\beta}_{n-1}} & R_1(W_n)}
$$
We must show that the right square commutes.  The left square
commutes by Lemma \ref{Transposition}.  The outer square commutes
since $f$ is a morphism.  Since $s_n$ is a natural isomorphism, it
follows that the right square commutes.
\end{proof}

\subsubsection{Adjointness of $E$ and $F$}

It will be necessary for us to know that $(E,F)$ is a bi-adjoint
pair.

\begin{theorem}
\label{AdjointPair} The functor $F:\M \rightarrow \M$ is right
adjoint to $E:\M \rightarrow \M$.
\end{theorem}
\begin{proof}

We need to check that there are natural isomorphisms:
$$
Hom_{\M}(E(V),W)\cong Hom_{\M}(V,F(W))
$$
for all $V,W \in \M$.  In other words, for $n \gg 0$ we must show
that there are natural isomorphism
\begin{equation}
\label{Adjoint of E and F} \chi_n=\chi_n(V,W):Hom_{G_n}(V_n, \ff^n
\otimes  W_n) \rightarrow Hom_{G_n}(R_1V_{n+1},W_n)
\end{equation}
for all $V_n,W_n \in \M_n$. It suffices to prove (\ref{Adjoint of E
and F}) for the case when $V$\ and $W$ are homogeneous.   Moreover,
we can assume that $V \in \M(k+1)$ and $W \in \M(k)$ for some $k$,
otherwise both sides of (\ref{Adjoint of E and F}) are zero.  Recall
the induction functors $I_0$ and $Pind_n^{n+1}$ (see Definitions
\ref{FunctorP_mn} and \ref{DefPind}). We have the following chain of
isomorphisms, which holds for $n \gg0$:
\begin{align*}\label{adunction_of_E_and_F}
Hom_{G_n}(R_1V_{n+1},W_n) &\cong Hom_{G_n}(RV_{n+1},W_n)
\\ &\cong Hom_{G_{n+1}}(V_{n+1},Pind_n^{n+1}W_n)
 \\ & \cong Hom_{G_{n+1}}(V_{n+1}, Sym(\ff^{n+1})\otimes I_{0}(W_n)) \\ &\cong Hom_{G_{n+1}}(V_{n+1}, \ff^{n+1} \otimes I_{0}(W_n))\\  &\cong Hom_{G_{n}}(V_{n}, \ff^{n} \otimes R_{0} (I_{0}W_n))) \\ & \cong Hom_{G_{n}}(V_{n}, \ff^{n} \otimes W_n).
\end{align*}

The first and fourth isomorphisms follow by degree considerations.
The second since $Pind_n^{n+1}$ is adjoint to restriction from
$G_{n+1}$ to $G_n$.  The third is from Lemma \ref{Pind decomp}. The
fifth isomorphism follows since $R_0$ defines an equivalence from
$\M_{n+1}(k)$ to $\M_n(k)$ for $n \gg 0$, and from the fact that
$R_0(\ff^{n+1} \otimes I_{0}(W_n)) \cong R_0(\ff^{n+1}) \otimes
R_0(I_{0}(W_n))$ since both are polynomial representations. Finally
the last map is an isomorphism by Corollary \ref{poly ind
inverse}.
\end{proof}

Our next task is to show that $F$ is also left adjoint to $E$.
Consider $V_n$ a rational representation of $G_n$. We let $V_n'$ be
the contravariant dual representation of $G_n$, i.e.,
$$(g\cdot\psi)(v)=\psi(g^tv), $$
where $\psi\in V_n^{*}$ and $g^t$ is the transpose of $g\in G_n$. It
is easy to see $V_n'$ is polynomial if and only if $V_n$ is
polynomial.  (This is why we use $g^t$ to define the action of $G_n$ on $V_n^*$, rather than the more commonly used $g^{-1}$.)

We define a duality functor $\mathbb{D}$ on $\M$, by sending
$V=(V_n,\alpha_n)$ to $V'=(V_n',\alpha_n')$, where
$$
\alpha_n'=(\alpha_n^{-1})^T
$$
and we are making the canonical identification
$R_0(V_{n+1}')=R_0(V_{n+1})'.$ It is a contravariant functor, and it
is easy to see that $\mathbb{D}\circ\mathbb{D}\simeq Id$. Moreover,
we have the following lemma.
\begin{lemma}{\label{Commutativity}}
There are natural isomorphisms
$$\mathbb{D}\circ E\circ \mathbb{D}\simeq E,$$ and $$\mathbb{D}\circ F\circ \mathbb{D}\simeq F.$$
\end{lemma}

\begin{proof}
We prove $\mathbb{D}\circ F \circ \mathbb{D} \simeq F$.  The other isomorphism can be checked directly or follows from the adjunction theorem between $F$ with $E$ below, and the fact that $\mathbb{D}$ is an involutive functor.  For $V_n\in\M_n$ there is an isomorphism $(V_n'\otimes \ff^n)'\simeq V_n\otimes \ff^n$.  Moreover, unwinding identifications one has $(\alpha_n'\otimes \iota_n)'=\alpha_n\otimes\iota_n$.  Therefore $\mathbb{D}\circ F \circ \mathbb{D} \simeq F$.
\end{proof}

\begin{theorem}
\label{Another-Adjoint-Pair} The functor $F:\M \rightarrow \M$ is
isomorphic to a left adjoint to $E:\M \rightarrow \M$
\end{theorem}

\begin{proof}
We need to construct a natural isomorphism $Hom_{\M}(F(V),W)\simeq
Hom_{\M}(V,E(W))$ for $V,W \in \M$. It is constructed by the
composition of following natural isomorphisms,
\begin{align*}
Hom_{\M}(F(V),W) &\simeq Hom_{\M}(\mathbb{D}(W),\mathbb{D}(F(V)))\\
&\simeq Hom_{\M}(\mathbb{D}(W),F(\mathbb{D}(V)))\\
&\simeq Hom_{\M}(E(\mathbb{D}(W),\mathbb{D}(V)))\\
&\simeq Hom_{\M}(\mathbb{D}(E(W)),\mathbb{D}(V))\\
&\simeq Hom_{\M}(V,E(W)).                        \\
&\end{align*} Here the first and fifth isomorphisms follow from the
self-duality of $\mathbb{D}$, the second and the fourth isomorphisms
follow from Lemma \ref{Commutativity}, and the third isomorphism is
from Theorem \ref{AdjointPair}.\qedhere

\end{proof}

Fix $V=(V_n,\alpha_n),W=(W_n,\beta_n) \in \M$. It will be useful for
us later on to have an explicit description of the isomorphism
$$\chi_n:Hom_{G_n}(V_n, \ff^n \otimes  W_n) \rightarrow
Hom_{G_n}(R_1(V_{n+1}),W_n)$$ from (\ref{Adjoint of E and F}). By
the proof of Theorem \ref{AdjointPair}, $\chi_n$ is the composition
$\chi_n = \tau_n \circ \kappa_n^{-1}$, where
$$\tau_n:Hom_{G_{n+1}}(V_{n+1}, \ff^{n+1} \otimes I_{0}(W_n))
\rightarrow Hom_{G_n}(R_1(V_{n+1}),W_n)$$ is the inverse map of the
composition  of the first four isomorphisms in the proof of Theorem
\ref{AdjointPair} and $$ \kappa_n:Hom_{G_{n+1}}(V_{n+1}, \ff^{n+1}
\otimes I_{0}(W_n)) \rightarrow Hom_{G_n}(V_{n}, \ff^{n} \otimes
W_n)$$ is the composition of the last two isomorphisms.

We now give explicit formulas for these morphisms.   By Lemma
\ref{Pind decomp}, we can view $\ff^{n+1} \otimes I_0(W_n)$ as a
subspace of functions from $M_{n+1}$ to $W_n$. Let $x_i$
($i=1,\cdots,n$) be the standard basis of $\ff^{n+1}$. Given a
morphism $g:V_{n+1}\rightarrow \ff^{n+1}\otimes I_0(W_n)$ and  $v\in
V_{n+1}$, write $g(v)=\sum_{i=1}^{n+1}x_i\otimes g_i$, where $g_i\in
I_0(W_n)$. We can view  $g_i$ as functions from $M_{n,n+1}$ to
$W_n$. Then given any matrix $A=(a_{ij} ) \in M_{n+1}$,
$$g(v)(A)=(\sum_{i=1}^{n+1}x_i\otimes
g_i)(A)=\sum a_{n+1,i}g_i(A_{n,n+1}),$$ where $A_{n,n+1}$ is the $n
\times (n+1)$ principal submatrix of $A$. Recall that  $I_{n,n+1}$
is the $n\times (n+1)$ matrix with ones on the main diagonal and
zeros elsewhere. Then we have the following lemma.

\begin{lemma} \label{Lemma for tau and kappa}
Given a morphism $g:V_{n+1}\rightarrow \ff^{n+1}\otimes I_0(W_n)$
and  $v\in V_{n+1}$, write $g(v)=\sum_{i=1}^{n+1}x_i\otimes g_i$ as
above.
\begin{enumerate}
\item If $v\in R_1(V_{n+1})$, then
$\tau_n(g)(v)=g_{n+1}(I_{n,n+1})$.
\item If $v\in R_0(V_{n+1})$, then
$\kappa_n(g)(\alpha_n(v))=\sum_{i=1}^{n}x_i\otimes g_i(I_{n,n+1})$.
\end{enumerate}
\end{lemma}
\begin{proof}

Note that $\tau_n$ is defined via Frobenius reciprocity, and so
$\tau_n(g)(v)=g(v)(I_{n+1,n+1})$.  Therefore
$\tau_n(g)(v)=(\sum_{i=1}^{n+1}x_i\otimes g_i)(I_{n+1,n+1})$, and
the first desired formula follows. For the second formula, by the
definition of $\kappa_n$, we have
\begin{align*}
\kappa_n(g)(\alpha_n(v)) &= (\pi_n\otimes ev_{I_{n,n+1}})(g(v)) \\
&= (\pi_n\otimes ev_{I_{n,n+1}})(\sum_{i=1}^{n+1}x_i\otimes g_i) \\
&= \sum_{i=1}^{n}x_i\otimes g_i(I_{n,n+1}).
\end{align*}
where $\pi_n$ is the natural $G_n$-equivariant projection
$\ff^{n+1}\to \ff^{n}$.
\end{proof}


\subsubsection{An endomorphism on $E$}
 In this section we construct an endomorphism $X$ on
$E$, i.e. a natural transformation $ X :E \rightarrow E$. This will
be used below to decompose $E$ into sub-functors $E_i$.

We consider the embedding $U(\mathfrak{gl}_n) \subset
U(\mathfrak{gl}_{n+1})$, analogous to the embedding of $G_n \subset
G_{n+1}$ defined above.  The Levi subgroup $G_n \times G_1$ of
$G_{n+1}$ acts on $U(\mathfrak{gl}_{n+1})$, by the restriction of
the adjoint action of $G_{n+1}$ on $U(\mathfrak{gl}_{n+1})$.

Set $$X_n=\sum_{i=1} ^{n} x_{n+1, i} x_{i, n+1} -n,$$ an element in
$U(\mathfrak{gl}_{n+1})$.

\begin{lemma}
\label{LemmaEpsilon} The element $X_n$ commutes with the adjoint
action of $G_n \times G_1$ on $ U(\mathfrak{gl}_{n+1})$, i.e.
$$
X_n \in U(\mathfrak{gl}_{n+1})^{G_n \times G_1}.
$$
\end{lemma}
\begin{proof}
Recall the Casimir element defined in Equation (\ref{Casimir}).  We compute:
\begin{eqnarray*}
C_{n+1}-C_n &=& \sum_{i=1} ^n x_{n+1 ,i} x_{i, n+1} + x_{i,
n+1}x_{n+1, i} + x_{n+1,n+1} ^2 \\ &=&\ 2X_n + (\sum_{i=1}^n
x_{i,i}) -nx_{n+1,n+1}+x_{n+1,n+1}^2+2n.
\end{eqnarray*}
Since $C_{n+1}-C_n$ and $(\sum_{i=1}^n x_{i,i})
-nx_{n+1,n+1}+x_{n+1,n+1}^2+2n$ commute with $G_n \times G_1$, the
result follows.
\end{proof}

By the above lemma, the action of $X_n$ on a representation $V_{n+1}
\in \M_{n+1}$ defines an endomorphism of the restriction functors
$R_i: \M_{n+1}\rightarrow \M_n$ for any $i \in \mathbb{Z}$.

\begin{proposition}
\label{PropactonE} Let $V=(V_n,\alpha_n) \in \M$.  The maps
$$X_n:R_1(V_{n+1}) \rightarrow R_1(V_{n+1})$$ glue to define a
morphism
$$
X(V):E(V) \rightarrow E(V),
$$
given by $X(V)=(X_n|_{R_1(V_{n+1})})$, so that $X :E \rightarrow E$
is a natural transformation.
\end{proposition}
\begin{proof}
Consider the following diagram,
$$
\xymatrix{ R_1(R_0(V_{n+1})) \ar[r]^{s_n} \ar[d]^{X_{n-1}} &
R_0(R_1(V_{n+1})) \ar[r]^>>>>>{\tilde{\alpha}_{n-1}}
\ar[d]^{X_n} & R_1(V_n) \ar[d]^{X_{n-1}} \\
R_1(R_0(V_{n+1})) \ar[r]^{s_n} & R_0(R_1(V_{n+1}))
\ar[r]^>>>>{\tilde{\alpha}_{n-1}} & R_1(V_n)}
$$
We want to show that the right square commutes for $n \gg 0$.  First
note that the fact that the outer square commutes follows since
$X_{n-1}$ is a natural transformation from $R_1$ to $R_1$ and by
definition: $\tilde{\alpha}_{n-1}=R_1(\alpha_n)\circ s_n^{-1}$.  The
following computation shows that the left square commutes:
\begin{eqnarray*}
(s_n ^{-1}\circ X_n \circ s_n )|_{R_1 \circ R_0(V_{n+1})} &=&
(\sum_{i=1} ^{n-1} x_{n,i}x_{i,n}+x_{n ,n+1}x_{n+1, n}-n)|_{R_1
\circ R_0(V_{n+1})} \\  &=& (X_{n-1}+x_{n, n+1}x_{n+1,
n}-1)|_{R_1\circ R_0(V_{n+1})} \\ &=& X_{n-1}|_{R_1\circ
R_0(V_{n+1})}.
\end{eqnarray*}
The last equality follows from Lemma \ref{RepLem1}.  Therefore the
right square commutes, which shows that $X(V) \in
Hom_{\M}(E(V),E(V))$.

Now suppose $V,W \in \M$ and $f \in Hom_{\M}(V,W)$.  Since $X_n$
acts on $R_1$, it follows that $X(W) \circ E(f) = E(f) \circ X(V)$.
Therefore $X:E \rightarrow E$ is a natural transformation.
\end{proof}

\subsubsection{An endomorphism of $F$}
We now construct an explicit natural transformation $Y$ of $F$,
related to $X:E \rightarrow E$ by adjunction.  This will be used
below to decompose $F$ into subfunctors $F_i$.

Fix $V,W \in \M$.  Let $h(X)$ be the morphism induced from
$X(V):E(V) \rightarrow E(V)$ by the functor $Hom_{\M}(\cdot,W)$.
Thus,
\begin{equation}
\label{Eq_hX} h(X):Hom_{\M}(E(V),W) \rightarrow Hom_{\M}(E(V),W).
\end{equation}
By adjunction (see Theorem \ref{AdjointPair}) we obtain an
endomorphism $h(X)^{\vee}$ of $Hom_{\M}(V,F(W))$ such that the
following diagram commutes:
\begin{equation*}
\xymatrix{ Hom_{\M}(V,F(W)) \ar[r]^{\chi}
\ar[d]_{h(X)^{\vee}} & Hom_{\M}(E(V),W) \ar[d]^{h({X})} \\
Hom_{\M}(V,F(W)) \ar[r]^{\chi} & Hom_{\M}(E(V),W) }
\end{equation*}
The map $\chi$ is shorthand for $(\chi_n(V,W))_{n=0}^{\infty}$ (see
Equation (\ref{Adjoint of E and F})). By Yoneda's Lemma we obtain a
natural transformation $X^{\vee}:F \rightarrow F$.

To describe $X^{\vee}$ more explicitly we introduce the following
explicitly defined natural transformation on $F$. Let $Y_n \in
U(\g_n)\otimes U(\g_n)$ be
$$
Y_n=\sum_{1 \leq i,j \leq n}x_{i,j} \otimes x_{j,i}
$$
and define $Y:F \rightarrow F$ by $Y=(Y_n)$.

\begin{proposition}
\label{Phi nat trans} The natural transformation $Y:F \rightarrow F$
is well-defined.
\end{proposition}
\begin{proof}
  Let $V=(V_n,\alpha_n) \in \M$.
Recall that $F(V)=(\ff^n \otimes V_n,\hat{\alpha}_n)$, where
$\hat{\alpha}_n=\iota_n\otimes \alpha_n$.  We need to show that the
following diagram commutes for $n \gg 0$:
\begin{equation}
\label{Diagram9} \xymatrix{ R_0(\ff^{n+1} \otimes V_{n+1})
\ar[r]^{Y_{n+1}} \ar[d]_{\hat{\alpha}_n} & R_0(\ff^{n+1} \otimes
V_{n+1}) \ar[d]^{\hat{\alpha}_n} \\ \ff^n \otimes V_n \ar[r]^{Y_n} &
\ff^n \otimes V_n}
\end{equation}

Let $x \otimes v \in R_0(\ff^{n+1} \otimes V_{n+1})$.  Then on the
one hand,
$$
\hat{\alpha}_n(x \otimes v)=x \otimes \alpha_n(v),
$$
and therefore,
\begin{equation}
\label{Eqtn9} Y_n(\hat{\alpha}_n(x \otimes v))=\sum_{1 \leq i,j \leq
n}x_{i,j} \cdot x \otimes x_{j,i} \cdot \alpha_n(v).
\end{equation}
On the other hand,
$$
Y_{n+1}(x \otimes v)=\sum_{1 \leq i,j \leq n+1}x_{i,j} \cdot x
\otimes x_{j,i} \cdot v.
$$
Since $x \otimes v$ is of degree zero for the action of $G_1$, this
implies that $x_{i,n+1} \cdot x=0$ for $i=1,...,n+1$, and moreover
that,
$$
\sum_{1 \leq i\leq n+1}x_{n+1,i} \cdot x \otimes x_{i,n+1} \cdot
v=0.
$$
Therefore,
\begin{equation}
\label{Eqtn10} \hat{\alpha}_n \circ Y_{n+1}(x \otimes v)=\sum_{1
\leq i,j \leq n}x_{i,j} \cdot x \otimes \alpha_n(x_{j,i} \cdot v).
\end{equation}
Since $\alpha_n$ is a $G_n$ morphism, it commutes with $x_{j,i}$ for
$1 \leq i,j \leq n$, and so (\ref{Eqtn9}) agrees with
(\ref{Eqtn10}), and diagram (\ref{Diagram9}) commutes.

It is trivial to show that $Y$ is compatible with morphisms $f:V
\rightarrow W$ in $\M$.\qedhere

\end{proof}

Fix $V,W \in \M$.  Let
\begin{equation}
\label{EqHdot} h(Y)^{\circ}:Hom_{\M}(V,F(W)) \rightarrow
Hom_{\M}(V,F(W))
\end{equation}
be the map induced by applying the functor $Hom_{\PP}(V,\cdot)$ to
the morphism $Y(W):F(W) \rightarrow F(W)$.

\begin{theorem}\label{Compatibility-Of-Endomorphisms}
In $End(F)$, $X^{\vee}=Y$.
\end{theorem}
\begin{proof}
Let $V,W \in \M$.  By the definition of $X^{\vee}$, to show the
equality $X^{\vee}=Y$ it suffices to show that the following diagram
commutes:
\begin{equation}
\label{DiagramJK} \xymatrix{ Hom_{\M}(V,F(W)) \ar[r]^{\chi}
\ar[d]_{h(Y)^{\circ}} & Hom_{\M}(E(V),W) \ar[d]^{h({X})} \\
Hom_{\M}(V,F(W)) \ar[r]^{\chi} & Hom_{\M}(E(V),W) }
\end{equation}

To show that diagram (\ref{DiagramJK}) commutes we need to show that
for $n \gg 0$ the following diagram commutes,
$$
\xymatrix{ Hom_{G_n}(V_n,\ff^n \otimes W_n) \ar[r]^{\chi_n}
\ar[d]_{h(Y_{n})^{\circ}} & Hom_{G_n}(R_1V_{n+1},W_n) \ar[d]^{h(X_n)} \\
Hom_{G_n}(V_n,\ff^n \otimes W_n) \ar[r]^{\chi_n} &
Hom_{G_n}(R_1V_{n+1},W_n) }
$$
Since $\chi_n = \tau_n \circ \kappa_n^{-1}$, to check that the above
diagram commutes we check that the following diagram commutes:
$$
\xymatrix@1{ Hom_{G_n}(V_n,\ff^n \otimes W_n)
\ar[d]_{h(Y_n)^{\circ}} & Hom_{G_{n+1}}(V_{n+1}, \ff^{n+1} \otimes
I_0(W_n)) \ar[l]_>>>>>>{\kappa_n}\ar[r]^<<<<{\tau_n}
\ar[d]_{h(Y_{n+1})^{\circ}} & Hom_{G_n}(R_1V_{n+1},W_n)
\ar[d]_{h(X_{n})}
\\
Hom_{G_n}(V_n,\ff^n \otimes W_n) & Hom_{G_{n+1}}(V_{n+1}, \ff^{n+1}
\otimes I_{0}(W_n)) \ar[l]_>>>>>>{\kappa_n} \ar[r]^<<<<{\tau_n} &
Hom_{G_n}(R_1V_{n+1},W_n) }
$$
By Lemma \ref{Lemma for tau and kappa}, we have explicit formulas
for $\kappa_n$ and $\tau_n$, which we will now use.

First we check the right diagram commutes, i.e., we need to check
$\tau_n( Y_{n+1}\circ g)=\tau_n(g)\circ X_n$, for any morphism
$g:V_{n+1}\rightarrow \ff^{n+1}\otimes I_{0}(W_n)$.

Let $v\in R_1(V_{n+1})$ and set $g(v)=\sum_{i=1}^{n+1}x_i\otimes g_i
$. Then we have
 \begin{eqnarray*}
 (Y_{n+1}\circ g)(v)
 &=& Y_{n+1}(g(v))\\ &=&\sum_{k,l=1}^{n+1}\sum_{i=1}^{n+1}x_{k,l}\cdot x_i\otimes x_{l,k}\cdot g_i \\
&=& \sum_{i,k=1}^{n+1} x_k\otimes x_{i,k}\cdot g_i.
\end{eqnarray*}
By Lemma \ref{Lemma for tau and kappa},  $\tau_n( Y_{n+1}\circ
g)(v)=\sum_{i=1}^{n+1}(x_{i,n+1}\cdot g_i)(I_{n,n+1})$. Then by
Equation (\ref{Eqtn6}) of Lemma \ref{Last Lemma}, $\tau_n(
Y_{n+1}\circ g)(v)=\sum_{i=1}^{n} (x_{i,n+1}\cdot g_i)(I_{n,n+1}).$

Now we will compute $(\tau_n(g)\circ X_n)(v)=\tau_n(g)(X_n(v))$. For
this, first we look at $g(X_n(v))$. Since $X_n$ commutes with $G_n$
on $R_1(V_{n+1})$, then
\begin{eqnarray*}
 g(X_n(v))&=&X_n(g(v))\\
&=&\sum_{k=1}^{n}x_{n+1,k}(x_{k,n+1}(\sum_{i=1}^{n+1} x_i\otimes g_i))-ng(v)\\ &=& nx_{n+1}\otimes g_{n+1}+ \sum_{k=1}^{n}(x_k\otimes x_{n+1,k}\cdot g_{n+1}+x_{n+1}\otimes x_{k,n+1}\cdot g_k \\
& &+\sum_{i=1}^{n+1}x_i\otimes x_{n+1,k}\cdot(x_{k,n+1}\cdot
g_i))-ng(v)
\end{eqnarray*}
Applying Lemma \ref{Lemma for tau and kappa} again, we have
 $$\tau_n(g)(X_n(v))=\sum_{k=1}^n (x_{k,n+1}\cdot g_k)(I_{n,n+1})+(x_{n+1,k}\cdot (x_{k,n+1}\cdot g_{n+1}))(I_{n,n+1}).$$

By Equation (\ref{Eqtn7}) of Lemma \ref{Last Lemma}, we have
$\tau_n(g)(X_n(v))=\sum_{k=1}^n (x_{k,n+1}\cdot g_k)(I_{n,n+1})$.
Hence the equality $\tau_n( Y_{n+1}\circ g)=\tau_n(g)\circ X_n$
holds. This shows that the right diagram commutes.

Finally we check that the left diagram also commutes. We have to check
that $$(h(Y_n)\circ \kappa_n)(g)=(\kappa_n\circ h(Y_{n+1}))(g).$$
Since $Y_{n+1}$ gives a natural transformation on the tensor functor
$\ff^{n+1}\otimes \cdot$, we have the  following commutative
diagram:
$$
\xymatrix{\ff^{n+1}\otimes W_{n+1} \ar[r]^{Y_{n+1}}  \ar[d]^{1 \otimes \beta_n^{\vee}} &  \ff^{n+1}\otimes W_{n+1} \ar[d]^{1 \otimes \beta_n^{\vee}} \\
\ff^{n+1}\otimes I_0(W_n) \ar[r]^{Y_{n+1}} & \ff^{n+1}\otimes
I_0(W_n)}
$$where $\beta_n^{\vee}$ is the morphism induced from $\beta_n:R_0(W_{n+1})\to W_n$. Combined with the commutativity of diagram (\ref{Diagram9}) and Lemma \ref{Diagram10}, the commutativity of left diagram follows.
\end{proof}

\subsubsection{The functors $E_a$ and $F_a$}
Finally we are ready to define the family of functors $\{E_a,F_a
\}$. For a vector space $U$ over $\ff$, an operator $T: U
\rightarrow U$ and a scalar $a \in \ff$, we denote by $U[a]$ the
$a$-generalized eigenspace of $T$ on $U$, i.e.
$$
U[a]=\bigcup_{N>0} ker(T-a)^N.
$$
For a morphism $f:V \rightarrow V$, where $V \in \M$, we can also
define the notion of a generalized eigenspace.  Indeed, if $f=(f_n)$
and $V=(V_n,\alpha_n)$, then we set $V[a]=(V_n[a],\alpha_n)$, where
$V_n[a]$ is the $a$-generalized eigenspace of $f_n$ on $V_n$.  It is
straightforward to check that the morphisms $\alpha_n$ restrict to
give gluing data $R_0(V_{n+1}[a]) \rightarrow V_n[a]$.

\begin{lemma}\label{LocalFinite}
Let $V \in \M$, $f \in End_{\M}(V)$, and $a \in \ff$.  Then there
exists $N>0$ such that
$$
V[a]=ker(f-a)^N.
$$
\end{lemma}

\begin{proof}
Any $V$ in $\M$ admits a composition series of finite length, and by
Lemma \ref{Schurslemma},  $f$ acts on each subquotient by some
scalar. Hence the lemma follows.
\end{proof}

\begin{definition}
For $V \in \M$ and $a \in \ff$ set $$E_a(V)=E(V)[a],$$ the
$a$-generalized eigenspace of $X(V):E(V) \rightarrow E(V)$.  This
defines a functor $$E_a: \M \rightarrow \M.$$
\end{definition}

It follows by Proposition \ref{PropactonE} that the functors $E_{a}$
are well-defined, and we now have a decomposition $E$ into
sub-functors:
$$
E=\bigoplus_{a \in \mathbb{\mathbb{F}}}E_a.
$$
We define sub-functors $F_a$ of $F$ completely analogously.
\begin{definition}
\label{TheFunctorFi} For $V \in \M$ and $a \in \ff$ set
$$F_a(V)=F(V)[a],$$ the $a$-generalized eigenspace of $Y:F(V)
\rightarrow F(V)$.  This defines a functor $$F_a: \M \rightarrow
\M.$$
\end{definition}

It follows by Proposition \ref{Phi nat trans} that the functors
$F_{a}$ are well-defined, and we now have a decomposition $F$ into
sub-functors:
$$
F=\bigoplus_{a \in \ff}F_a.
$$

\begin{proposition}\label{Sub-Adjoint-pair}
For every $a$, $(E_a,F_a)$ is an adjoint pairs of functors.
\end{proposition}
\begin{proof}
Given objects $V$ and $W$ in the category $\M$,  by the definition
of $E_a$ and $F_a$ and Lemma \ref{LocalFinite}, there exists a
positive integer $N$ such that,
$$E_a(V)=ker(X-a)^N:E(V)\rightarrow E(V)),$$
$$F_a(W)=ker(Y-a)^N:F(W)\rightarrow F(W) ).$$
 Recall that we defined morphisms $h$ and $h^0$ in (\ref{Eq_hX}) and (\ref{EqHdot}).  Consider the following diagram:

\begin{equation}
\label{DiagramJ} \xymatrix{
Hom_{\M}(V,F_a(W))\ar[r] \ar[d]_{h(ker(Y-a)^N)^{\circ}} & Hom_{\M}(E_{a}(V),W)\ar[d]^{h (ker(X-a)^N)}  \\
 Hom_{\M}(V,F(W)) \ar[r]^{\chi}
\ar[d]_{h((Y-a)^N)^{\circ}} & Hom_{\M}(E(V),W) \ar[d]^{h((X-a)^N)} \\
Hom_{\M}(V,F(W)) \ar[r]^{\chi} & Hom_{\M}(E(V),W) }
\end{equation}
To be precise, for example $h (ker(X-a)^N)$ is obtained by applying $Hom_\M(\cdot,W)$ to the morphism $ker(X-a)^{N}:E(V)\to E(V)$, and then restricting the resulting morphism to $Hom_\M(E_a(V),W)$. In the above diagram, by Theorem \ref{AdjointPair}, $\chi$ is an
isomorphism. By Theorem \ref{Compatibility-Of-Endomorphisms}, the
bottom square commutes. Since $E_a(V)$ is a direct summand of
$E(V)$, $h(ker(X-a)^N))$ is the kernel of $h((X-a)^N)$. So the
isomorphism $\chi$ induces  an isomorphism from $Hom(V,F_a(W))$ to
$Hom(E_a(V),W)$.
\end{proof}

\section{Categorifying the Fock space}
In this section we prove our main theorem.  First, in Section
\ref{Sec4.1}, we recall the definitions of the degenerate affine
Hecke algebra and Chuang and Rouquier's notion of categorification.
In Section \ref{Sec4.2} we define the data of a
$\g$-categorification on $\M$ and prove our main result, namely that
this data is indeed a $\g$-categorification which categorifies the
Fock space representation of $\g$.  By Chuang-Rouquier theory we
deduce derived equivalences between blocks of $\M$.  In Section
\ref{SectionCrystal}  we recover the crystal of Fock space from the
set of simple objects of $\M$.

\subsection{The definition of $\g$-categorification}
\label{Sec4.1} We recall Chuang and Rouquier's notion of
$\g$-categorification following \cite{R}.

\begin{definition}
Let $\bar{H}_n$ be the degenerate affine Hecke algebra of $GL_n$. As
an abelian group
$$
\bar{H}_n=\Z[y_1,...,y_n]\otimes \Z \mathfrak{S}_n.
$$
The algebra structure is defined as follows: $\Z[y_1,...,y_n]$ and
$\Z \mathfrak{S}_n$ are subalgebras, and the following relations
hold between the generators of these subalgebras:
$$
\tau_iy_j=y_j\tau_i \text{ if }|i-j| \geq 1
$$
and
\begin{equation}
\label{HeckeRelation} \tau_iy_{i+1}-y_i\tau_i=1
\end{equation}
(here $\tau_1,...,\tau_{n-1}$ are the simple transpositions of $\Z
\mathfrak{S}_n$).
\end{definition}

For an abelian $\ff$-linear category $\mathcal{V}$, let
$K(\mathcal{V})$ denote the \emph{complexified} Grothendieck group
of $\mathcal{V}$.  The equivalence class of an object $V \in
\mathcal{V}$ is denoted $[V]\in K(\mathcal{V})$, and given an exact
functor $F:\mathcal{V} \rightarrow \mathcal{V}$, $[F]:
K(\mathcal{V}) \rightarrow K(\mathcal{V})$ denotes the induced
linear operator.
\begin{definition}[Definition 5.29, \cite{R}]
\label{DefinitionCategorification} Let $\mathcal{V}$ be an abelian
$\ff$-linear category.  A $\g$-categorification on $\mathcal{V}$ is
the data of:
\begin{enumerate}
\item an adjoint pair $(E,F)$ of exact functors $\mathcal{V} \rightarrow \mathcal{V,}$
\item $X \in End(E)$ and $T \in End(E^2),$
\item a decomposition $\mathcal{V}=\bigoplus_{\omega \in P} \mathcal{V}_{\omega}$.
\end{enumerate}
Let $X^{\vee} \in End(F)$ be the endomorphism of $F$ induced by
adjunction.  Then given $i \in \ff$ let $E_i$ (resp. $F_i$) be the
generalized $i$-eigensubfunctor of $X$ (resp. $X^{\vee}$) acting on
$E$ (resp. $F$).  We assume that
\begin{enumerate}
\item[(4)] $E=\bigoplus_{i \in \Z/p\Z} E_i$,
\item[(5)] The action of $\{ [E_i],[F_i]\}_{i \in \Z/p\Z}$ on $K(\mathcal{V})$ gives rise to an integrable representation of $\g,$
\item[(6)] For all $i$, $E_i(\mathcal{V}_{\omega}) \subset \mathcal{V}_{\omega+\alpha_i}$ and $F_i(\mathcal{V}_{\omega}) \subset \mathcal{V}_{\omega-\alpha_i}$,
\item[(7)] $F$ is isomorphic to the left adjoint of $E$,
\item[(8)] The degenerate affine Hecke algebra $\bar{H}_n$ acts on $End(E^n)$ via
\begin{equation}
\label{EqXX} y_i \mapsto E^{n-i}XE^{i-1} \text{ for }1 \leq i \leq
n,
\end{equation}
and
\begin{equation}
\label{EqTT} \tau_i \mapsto E^{n-i-1}TE^{i-1} \text{ for }1 \leq i
\leq n-1.
\end{equation}
\end{enumerate}

\end{definition}

\begin{remark}
To clarify notation, the natural endomorphism $y_i$ of $E^n$ assigns
to $V \in \mathcal{V}$ an endomorphism of $E^n(V)$ as follows: apply
the functor $E^{n-i}$ to the morphism $$X(E^{i-1}(V)):E^i(V)
\rightarrow E^i(V).$$
\end{remark}


\begin{remark}
\label{RmrkRuoq1} Rouquier defines a $2$-Kac Moody algebra
$\mathfrak{A}(\g)$ associated to $\g$, and shows that a
$\g$-categorification $\mathcal{V}$ is equivalent to a
$2$-representation of  $\mathfrak{A}(\g)$, i.e. a $2$-functor
$\mathcal{V}:\mathfrak{A}(\g) \rightarrow Cat$  (cf. Theorem 5.30 in
\cite{R}).
\end{remark}

\subsection{The $\g$-categorification on $\M$}
\label{Sec4.2} By now we've defined functors $E$ and $F$ on $\M$
(cf. Section \ref{SectionFunctors}), and shown that these are
bi-adjoint (Theorems \ref{AdjointPair} and
\ref{Another-Adjoint-Pair}).  We also have a natural endomorphism $X
\in End(E)$ (Proposition \ref{PropactonE}), and we've shown that
$X^{\vee}=Y$ (cf. Theorem \ref{Compatibility-Of-Endomorphisms}). Now
we introduce the remaining data necessary to define a
$\g$-categorification, and finally prove our main theorem.

\subsubsection{The $\g$-action on $K(\M)$}
The vector space $K(\M)$ has a basis $[V(\lambda)]$, where $\lambda
\in \Lambda$.  Therefore we have a natural linear isomorphism
$$
\xi:K(\M) \rightarrow \FF
$$
given by $\xi([V(\lambda)])=v_{\lambda}$, where $\FF$ is the Fock
space representation of $\g$ (cf. Section \ref{Fock space}).
  In this section we show that $\xi$ intertwines the operators $[E_i]$ and $[F_i]$ acting on $K(\M)$ with the Chevalley generators  $e_i$ and $f_i$ acting on $\FF$.  Consequently the operators $[E_i]$ and $[F_i]$ induce an action of $\g$ on $K(\M)$, and $\xi$ is an isomorphism of $\g$-modules.

\begin{lemma}
Suppose $V_k \in \M_k(k)$ has a Weyl filtration
$$
0=V_k^0 \subset V_k^1 \subset \cdots \subset V_k^N=V_k,
$$
where for $i=0,...,N-1$
$$
V_k^{i+1}/V_k^i \simeq V_k(\mu_i).
$$
Then $V=\Psi^{-1}(V_k) \in \M(k)$ has a Weyl filtration
$$
0=V^0 \subset V^1 \subset \cdots \subset V^N=V
$$
where for $i=0,...,N-1$.
$$
V^{i+1}/V^i \simeq V(\mu_i).
$$
\end{lemma}

\begin{proof}
Note that $V_k^i \in \M_k(k)$ for $i=0,...,N$.  Set
$$
V^i=\Psi^{-1}(V_k^i).
$$
By definition $V^i \in \M(k)$, and we have a filtration
$$
0=V^0 \subset V^1 \subset \cdots \subset V^N=V.
$$
Now $\Psi(V^{i+1}/V^i) \simeq V_k^{i+1}/V_k^i \simeq V_k(\mu_i)$.
Therefore by Propositions \ref{InvertoS} and \ref{WeylObjects},
$$
V^{i+1}/V^i \simeq V(\mu_i).
$$
\end{proof}

\begin{proposition}
\label{PropWeylFiltration} Let $\lambda \in \Lambda$ and set
$V=V(\lambda)$.  Then,
\begin{enumerate}
\item The object $E(V)$ admits a Weyl filtration
$$
0=E(V)^0 \subset E(V)^1 \subset \cdots \subset E(V)^N=E(V).
$$
The composition factors that occur in this filtration are isomorphic
to $V(\mu)$ for all $\mu \in \Lambda$ such that $ \xymatrix{ \mu
\ar[r] & \lambda} $, and each such factor occurs exactly once.
\item The object $F(V)$ admits a Weyl filtration
$$
0=F(V)^0 \subset F(V)^1 \subset \cdots \subset F(V)^N=F(V).
$$
The composition factors that occur in this filtration are isomorphic
to $V(\mu)$ for all $\mu \in \Lambda$ such that $ \xymatrix{ \lambda
\ar[r] & \mu} $, and each such factor occurs exactly once.
\end{enumerate}
\end{proposition}

\begin{proof}
Let $k$ be such that $\lambda \in \Lambda_k$.  We have the following
diagram, which commutes:
$$
\xymatrix{ \M(k) \ar[r]^{\Psi_k} \ar[d]_{E} & \M_k(k) \ar[d]^{R_1}
\\ \M(k-1) \ar[r]^{\Psi_{k-1}} & \M_{k-1}(k-1) }
$$
Note that by Proposition \ref{Forgetful-Functor} the functors
$\Psi_k$ and $\Psi_{k-1}$ are equivalences. Write $V=(V_n) \in
\M(k)$.  Then, by this commutative square and Proposition
\ref{InvertoS}, $$V \simeq \Psi^{-1}(R_1(V_k)).$$ By the above lemma
and Lemma \ref{Weyl-Subquotient}, part (1) of the proposition
follows.  The proof of part (2) of the proposition is entirely
analogous with the above square replaced by
$$
\xymatrix{ \M(k) \ar[r]^{\Psi_{k+1}} \ar[d]_{F} & \M_{k+1}(k)
\ar[d]^{\ff^{k+1}\otimes\ \cdot} \\ \M(k+1) \ar[r]^{\Psi_{k+1}} &
\M_{k+1}(k+1) }
$$
\end{proof}



 \begin{lemma}
 \label{LemmaDecomposition}
 Let $\lambda \in \Lambda$ and $V=V(\lambda) \in \M$, and consider the Weyl filtrations of $E(V)$ and $F(V)$ as in Proposition \ref{PropWeylFiltration}.  Then $X$ (resp. $Y$) preserves the filtration of $E(V)$ (resp. $F(V)$).  Moreover,
\begin{enumerate}
\item Given $0 \leq i \leq N-1$ set $\mu \in \Lambda,j\in \Z / p\Z$ such that $E(V)^{i+1}/E(V)^i \simeq V(\mu)$ and $
\xymatrix{ \mu \ar[r]^{j} & \lambda} $.  Then $X$ acts on
$E(V)^{i+1}/E(V)^i$ by $j$.
\item Given $0 \leq i \leq N-1$ set $\mu \in \Lambda,j\in \Z / p\Z$ such that $F(V)^{i+1}/F(V)^i \simeq V(\mu)$ and $
\xymatrix{ \lambda \ar[r]^{j} & \mu} $.  Then $Y$ acts on
$F(V)^{i+1}/F(V)^i$ by $j$.
 \end{enumerate}
 In particular, $E_j=F_j=0$ for all $j \in \ff$ such that $j \not\in \Z / p\Z$, i.e.
$$
E=\bigoplus_{j \in \Z / p\Z}E_j
$$
and
$$
F=\bigoplus_{j \in \Z / p\Z}F_j.
$$
 \end{lemma}

 \begin{proof}
 It is clear from the formula
\begin{equation}
\label{EqX} C_{n+1}-C_n= 2X_n + (\sum_{i=1}^n x_{i,i})
-nx_{n+1,n+1}+x_{n+1,n+1}^2+2n
\end{equation}
that $X$ preserves the filtration of $E(V)$.  Now let $V=(V_n)$. Set
$k$ so that $\lambda \in \Lambda_k$ and let $n \geq k$. Consider
first a Weyl filtration of $R_1(V_{n+1})$:
$$
0=R_1(V_{n+1})^0 \subset R_1(V_{n+1})^1 \subset \cdots \subset
R_1(V_{n+1})^N=R_1(V_{n+1})
$$
such that
$$
R_1(V_{n+1})^{i+1}/R_1(V_{n+1})^i \simeq V_{n}(\mu).
$$
(cf. Lemma \ref{Weyl-Subquotient}). Since $ \xymatrix{ \mu
\ar[r]^{j} & \lambda} $ there exists $\ell$ such that
$\lambda_{\ell}-\ell=j$.
 By (\ref{CasScalar}) one computes
 $$
 c_{n+1}(\lambda)-c_n(\mu)=2(\lambda_{\ell}-\ell)+k+n.
 $$
Since the degree of representation $V_{n+1}$ is $k$, which is the size of the partition $\lambda$, then $\sum_{i=1}^{n+1} x_{i,i}$ acts by $k$ on $V_{n+1}$. On the other hand, by the definition of $R_1$, $x_{n+1}$ acts on $R_1(V_{n+1})$ by $1$.
Then by (\ref{EqX})
 it follows that $X_n$ acts on $R_1(V_{n+1})^{i+1}/R_1(V_{n+1})^i\simeq V_n(\mu)$ by
 $$(c_{n+1}(\lambda)-c_n(\mu)-k-n)/2=\lambda_\ell-\ell=j . $$  Hence $X$ acts on $E(V)^{i+1}/E(V)^i$ by $j$.  This proves the statement (1) of the lemma.

The statement (2) of the lemma follows along similar lines. Firstly,
it is clear that $Y$ preserves the filtration of $F(V)$. Now
considering $Y_n$ as an element of $U(\g_n)\otimes U(\g_n)$, note
that
\begin{equation}
\label{EqCoproduct} Y_n=\frac{1}{2}(\Delta(C_n)-C_n \otimes 1 -1
\otimes C_n)
\end{equation}
where $\Delta:U(\g_n) \rightarrow U(\g_n)\otimes U(\g_n)$ is the
coproduct in $U(\g_n)$.  Consider a Weyl filtration of
$F(V)_n=\ff^n\otimes V_{n}$:
$$
0=F(V)^0_n \subset F(V)^1_n \subset \cdots \subset F(V)^N_n=F(V)_n
$$
such that
$$
F(V)^{i+1}_n/F(V)^i_n \simeq V_{n}(\mu).
$$By (\ref{EqCoproduct}) and Lemmas
\ref{Weyl-Subquotient} and \ref{SCALAR}, it follows that $Y_n$ acts
on $F(V)^{i+1}_n/F(V)^i_n$ by $j$.

The last statement of the lemma follows from parts (1) and (2) and
the fact that Weyl objects descend to a basis of the Grothendieck
group $K(\M)$.
 \end{proof}

 Recall that $\xi:K(\M) \rightarrow \FF$ is defined by $\xi([V(\lambda)])=v_{\lambda}$.  As an immediate corollary of the above lemma we obtain:

\begin{proposition}
\label{Ecommutative} For every $i \in \Z / p\Z$ the following
diagram commutes:
$$
\xymatrix{ K(\M) \ar[r]^{\xi} \ar[d]_{[E_i]} & \FF \ar[d]^{e_i} \\
K(\M) \ar[r]^{\xi} & \M }
$$
Similarly, there is a commutative square with $[E_i]$ and $e_i$
replaced by $[F_i]$ and $f_i$, respectively.

In particular, the operators $[E_i]$ and $[F_i]$ define an action of
$\g$ on $K(\M)$ via $e_i \mapsto [E_i]$ and $f_i \mapsto [F_i]$, and
$\xi$ is an isomorphism of $\g$-modules.

\end{proposition}

\subsubsection{The $\bar{H}_n$-action on $End(E^n)$}
To define the action of the degenerate affine Hecke algebra on
powers of $E$ we first need to define the operator $T$ on $E^2$. Set
$t_n=s_{n+1}$.  Clearly for $V_{n+2} \in \M_{n+2}$, $t_n$ defines an
operator on $R_1^2(V_{n+2})$.  Let $T=(t_n)_{n=0}^{\infty}$.

\begin{lemma}
\label{PropT} The operator $T$ acts on $E\circ E$, i.e. $T \in
End(E^2)$.
\end{lemma}

\begin{proof}
Given $V=(V_n, \alpha_n)$, for $n\gg 0$, we need to check the
following diagram commutes:
$$
\xymatrix{R_0(E^2(V)_n) \ar[r]^{\tilde{\tilde{\alpha}}_{n-1}}
\ar[d]^{t_{n}} & E^2(V)_{n-1} \ar[d]^{t_{n-1}} \\
R_0(E^2(V)_n) \ar[r]^{\tilde{\tilde{\alpha}}_{n-1}} & E^2(V)_{n-1} }
$$
Thus it suffices to check that the following diagram commutes:
$$
\xymatrix{
R_{0} \circ R_1 \circ R_1(V_{n+1})\ar[rr]^>>>>>>>>>{\alpha_{n+1}\circ s_{n+1}^{-1}\circ s_n ^{-1}} \ar[d]^{s_{n+1}} && R_1 \circ R_1(V_{n+1}) \ar[d]^{s_n} \\
R_{0} \circ R_1 \circ R_1(V_{n+2})
\ar[rr]^>>>>>>>>>{\alpha_{n+1}\circ s_{n+1}^{-1}\circ s_n ^{-1}} &&
R_1 \circ R_1(V_{n+1}) }
$$
This diagram  commutes by the braid relation, $ s_n
s_{n+1}s_n=s_{n+1}s_ns_{n+1}$, and the fact that $\alpha_{n+1}$
commutes with $s_n$ ($\alpha_{n+1}$ is a morphism of
$G_{n+1}$-modules).\qedhere

\end{proof}

\begin{proposition}
\label{HeckeAction} The degenerate affine Hecke algebra $\bar{H}_n$
acts on $End(E^n)$ via formulas (\ref{EqXX}) and (\ref{EqTT}).
\end{proposition}
\begin{proof}
First note that for $V_{n+2} \in \M_{n+2}$, $t_n^2$ acts on
$R_1^2(V_{n+2})$ by the identity.  Therefore $T^2=1$ in $End(E^2)$.
The only other relation that is not trivial to show is relation
(\ref{HeckeRelation}). Relation (\ref{HeckeRelation}) is a
consequence of the following equality in $End(E^2)$:
$$
T\circ XE-EX\circ T=1.
$$
This equality follows from the identity
\begin{equation}
\label{eq2} s_nX_{n-1}-X_ns_n=1
\end{equation}
in $End(R_1^2)$.  Since we have,
$$
s_nX_ns_{n}^{-1}=X_{n-1}+x_{n,n+1}x_{n+1,n}-1
$$
equation (\ref{eq2}) follows from Lemma \ref{RepLem2}.
\end{proof}

\subsubsection{The decomposition of $\M$ as a direct sum of weight categories}
\label{DecompositionM}

Let
$\mathcal{F}=\bigoplus_{\omega\in P} \mathcal{F}_\omega$ be the
weight space decomposition as a representation of $\g$. Recall that for a partition
$\lambda \in \Lambda$, $v_\lambda$ is a weight vector. We
define a \emph{weight function} $wt: \Lambda\rightarrow P$ by requiring that
$v_\lambda\in \mathcal{F}_{wt(\lambda)}$.

We recall some combinatorial notions.  For a nonnegative integer $d$, let $\Lambda_d$ denote the set of partitions of $d$.  A partition $\lambda$ is a \emph{$p$-core} if there exist no $\mu \subset \lambda$ such that the skew-partition $\lambda/\mu$ is a rim $p$-hook.  By definition, if $p=0$ then all partitions are $p$-cores.  Given a partition $\lambda$, we denote by $\tilde{\lambda}$ the $p$-core obtained by successively removing all rim $p$-hooks. We define the number $(|\lambda|-|\tilde{\lambda}|)/p$ to be the p-weight of $\lambda$.

Define an equivalence relation $\sim$ on $\Lambda_d$ by decreeing $\lambda \sim \mu$ if $\tilde{\lambda}=\tilde{\mu}$.

 Let $\lambda,\mu\in\Lambda_{d}$.  As a consequence of (11.6) in \cite{Kl05} we have
\begin{equation}\label{KleshEqua}
\tilde{\lambda}=\tilde{\mu} \Longleftrightarrow wt(\lambda)=wt(\mu).
\end{equation}
Therefore we index the set of equivalence classes $\Lambda_d/\sim$ by weights in $P$, i.e. a weight $\omega \in P$ corresponds to a subset (possibly empty) of $\Lambda_d$.

Let $Irr\M_d$ denote the set of simple objects in $\M_d$ up to isomorphism.  This set is naturally identified with $\Lambda_d$.  We say two simple objects in $\M_d$ are \emph{adjacent} if they occur as composition factors of some indecomposable object in $\M_d$.    Consider the equivalence relation $\approx$  on $Irr\M_d$ generated by adjacency.  Via the identification of $Irr\M_d$ with $\Lambda_d$ we obtain an equivalence relation $\approx$ on $\Lambda_d$.

\begin{theorem}[Theorem 2.12, \cite{D}]
The equivalence relations $\sim$ and $\approx$ on $\Lambda_d$ are the same.
\end{theorem}

By the above theorem, Equation (\ref{KleshEqua}) and Equation (\ref{weight}) we can label any block of $\M$ by weights $\omega \in P$, and the p-weight of a block is well-defined.  So we have the decomposition,

\begin{equation}
\label{Decomposition}
\M=\bigoplus_{\omega} \M_\omega.
\end{equation}

\subsubsection{The $\g$-categorification on $\M$}
We can now state and prove our main result:
\begin{theorem}
\label{CatTheorem} The data of
\begin{enumerate}
\item the adjoint pair of functors $(E,F)$,
\item $X \in End(E)$ and $T \in End(E^2)$, and
\item the decomposition of $\M = \bigoplus_{\omega \in P}\M_{\omega}$
\end{enumerate}
is a $\g$-categorification on the abelian $\ff$-linear category
$\M$.
\end{theorem}

\begin{proof}
The adjointness of $(E,F)$ is Theorem \ref{AdjointPair}.  The
endomorphism $X$ on $E$ is defined in Proposition \ref{PropactonE},
while $T$ is defined in Lemma \ref{PropT}.  The decomposition of
$\M$ into subcategories is Equation (\ref{Decomposition}).

Now we must check that conditions (4)-(8) of Definition
\ref{DefinitionCategorification} are satisfied.  Condition (4)
follows from Lemma \ref{LemmaDecomposition}.  Now, since
$X^{\vee}=Y$ (cf. Theorem \ref{Compatibility-Of-Endomorphisms}), the
functors $F_i$ we defined (cf. Definition \ref{TheFunctorFi}) agree
with the functors $F_i$ that arise as generalized eigenspaces of
$X^{\vee}$ acting on $F$.  Therefore, conditions (5) and (6) are a consequence
of Proposition \ref{Ecommutative}.  Condition (7) is Theorem
\ref{Another-Adjoint-Pair}, while condition (8) is Proposition
\ref{HeckeAction}.
\end{proof}

By \cite[Section 12]{Ka},  any weight $\omega$ appearing in Fock space
is of the form $\sigma(\omega_0)-\ell \delta$, where $\omega_0$ is
the first fundamental weight and $\sigma$ is some element in the
affine Weyl group of $\g$.  By \cite[Proposition, 11.1.5]{Kl05},
$\ell $ is exactly the p-weight of the corresponding block.
Therefore the weights of any two blocks are conjugate by some
element of the affine Weyl group if and only if  they have the same
p-weight. As a consequence of Chuang-Rouquier theory we obtain:

\begin{corollary}\label{Derived_Equivalence}
If two blocks of $\M$ have the same p-weight then
they are derived equivalence.
 \end{corollary}

\subsection{The Misra-Miwa crystal from the category $\M$}
\label{SectionCrystal}

We now show how to recover the Misra-Miwa crystal of Fock space (cf.
\cite{MM}) from the category $\M$ . We first briefly recall the
definition of this crystal.  For this we need to first also recall
several combinatorial notions (see \cite{BK} for a more thorough
discussion of these notions).

Label all the $i$-addable boxes of a partition $\lambda$ by $+$ and
all $i$-removable boxes by $-$. The \emph{$i$-signature} of
$\lambda$ is the sequence of $+$ and $-$ obtained by going along the
rim of the Young diagram from bottom left to top right and reading
off all the signs. The \emph{reduced $i$-signature} of $\lambda$ is
obtained from the $i$-signature by successively erasing all
neighboring pairs of the form $+-$. Note the reduced $i$-signature
always looks like a sequence of $-$'s followed by $+$'s. Boxes
corresponding to a $-$ in the reduced $i$-signature are called
\emph{$i$-normal}, boxes corresponding to a $+$ are called
\emph{$i$-conormal}. The rightmost $i$-normal box (corresponding to
the rightmost $-$ in the reduced $i$-signature) is called
\emph{$i$-good}, and the leftmost $i$-conormal box (corresponding to
the leftmost $+$ in the reduced $i$-signature) is called
\emph{$i$-cogood}.

The Misra-Miwa crystal of Fock space (cf. \cite{LLT}) consists of
the set $\Lambda$ together with maps
$$
wt:\Lambda \rightarrow P, \tilde{e}_i,\tilde{f}_i:\Lambda
\rightarrow \Lambda \cup\{0\}, \epsilon_i,\phi_i:\Lambda \rightarrow
\Z
$$
defined as follows:
\begin{enumerate}
\item The $wt$ function is defined above in Section \ref{DecompositionM}.
\item The operator $\tilde{e}_i$ is given by the rule: if there exists an  $i$-good box for $\lambda$, then $\tilde{e}_i(\lambda)=\mu$, where $\mu$ is  obtained by removing this $i$-good box  from  $\lambda$;  otherwise $\tilde{e}_i(\lambda)=0$.
\item The operator $\tilde{f}_i$ is given by the rule: if there exists an $i$-cogood box for $\lambda$, then $\tilde{f}_i(\lambda)=\mu$ where $\mu$ is obtained from $\lambda$ by adding this $i$-cogood box;  otherwise $\tilde{f}_i(\lambda)=0$.
\item$\epsilon_i(\lambda)$ is the number of $i$-normal boxes of $\lambda$
\item$\phi_i(\lambda)$ is the number of $i$-conormal boxes of $\lambda$.
\end{enumerate}

We now reformulate a theorem of Brundan and Kleshchev \cite{BK} in
our setting.
\begin{theorem} For any simple object $L(\lambda)$ in $\M$, if  $\lambda$ has an $i$-cogood box, then the socle of  $F_i(L(\lambda)) $ is $L(\mu)$, where $\mu$ is obtained from $\lambda$ by adding the $i$-cogood box.  Otherwise $F_i(L(\lambda))=0$.
\end{theorem}
\begin{proof}
The functor $F:\M\rightarrow \M$ is given by $F=(F_n)$, where
$F_n:\M_n\rightarrow \M_n$ is tensoring with the standard module
$\ff^n$. Similarly, $F_i=((F_i)_n)$, where $(F_i)_n$ is the
generalized $i$-eigen-subfunctor of $Y_n$.

It suffices to show that given a partition $\lambda$, then for $n
\gg 0$ the socle of $(F_i)_n(L_n(\lambda))$ is
$L_n(\lambda+\epsilon_k)$ if the box $(k,\lambda_k+1)$ is cogood and
$\lambda_k +1-k=i$ (here $\epsilon_k$ denotes the weight
$(0,...,1,...,0)$, where the one is in the $k^{th}$ position), and
otherwise it is $0$. In \cite{BK}, the translation functor $Tr^i$ is
introduced, and by Theorem $A(i)$ \cite{BK}, it is easy to see that
$Tr^i$ coincides with $(F_i)_n$. Then Theorem $B(i)$ in \cite{BK}
implies our theorem.
\end{proof}


\begin{corollary}
For any simple object $L(\lambda)$ in $\M$, if  $\lambda$ has an
$i$-good box, then the socle of  $E_i(L(\lambda)) $ is $L(\mu)$,
where $\mu$ is obtained from $\lambda$ by deleting the $i$-good box.
Otherwise $E_i(L(\lambda))=0$.
 \end{corollary}
\begin{proof}
Consider two simple objects $L(\lambda)$ and $L(\mu)$. By
Proposition \ref{Sub-Adjoint-pair}, we have
\begin{equation*}
Hom_{\M}(E_i(L(\lambda)),L(\mu))=Hom_{\M}(L(\lambda),F_i(L(\mu))).
\end{equation*}
Recall that on the category $\M$, there is a contravariant duality
$\mathbb{D}$, which maps any simple object to itself.  By Lemma
\ref{Commutativity}, $\mathbb{D}$ commutes with $E$ (and hence with
$E_i$), so therefore
$$Hom_{\M}(E_i(L(\lambda)),L(\mu))=Hom_{\M}(L(\mu),E_i(L(\lambda)))$$
and so we have
\begin{equation}
\label{EqCrystal}
Hom_{\M}(L(\lambda),F_i(L(\mu)))=Hom_{\M}(L(\mu),E_i(L(\lambda))).
\end{equation}
Now note that $\mu$ is obtained from $\lambda$ by deleting the
$i$-good box if and only if $\lambda$ is obtained from $\mu$ by
adding this box (which is $i$-cogood for $\mu$). Therefore, by the
above theorem, Schur's lemma (Lemma \ref{Schurslemma}), and equation
(\ref{EqCrystal}), the corollary follows.
%
%
%
\end{proof}

Let $\mathbb{B}$ be the set of isomorphism classes of simple objects
in $\M$, namely $$\mathbb{B}=\{L(\lambda):\lambda\in \Lambda \}.$$
By the above theorem and corollary, we can define operators
$$\tilde{E}_i,\tilde{F}_i:\mathbb{B} \rightarrow \mathbb{B}\cup
\{0\}$$ as follows: for any simple object $L(\lambda)$, set
$\tilde{E}_i(L(\lambda))$ be the socle of $E_i(L(\lambda))$, and set
$\tilde{F}_i(L(\lambda))$ be the socle of $F_i(L(\lambda))$.
Moreover, let
$$\varepsilon_i(L(\lambda))=max\{m:\tilde{E}_i^m(L(\lambda))\not =0
\}$$ and
$$\varphi_i(L(\lambda))=max\{m:\tilde{F}_i^m(L(\lambda))\not =0\}$$
Finally, set $wt(L(\lambda))=wt(\lambda)$. By the above theorem and
corollary, we obtain:
\begin{theorem}\label{Crystal}
The data $(\mathbb{B},\tilde{E}_i,\tilde{F}_i, \varepsilon_i,
\varphi_i, wt)$ defines the crystal of Fock space, which is
isomorphic to the Misra-Miwa crystal
$(\Lambda,\tilde{e}_i,\tilde{f}_i,\epsilon_i,\phi_i, wt)$ described
above.
\end{theorem}

\appendix

\section{The general linear group}

In this appendix we collect some standard/technical results that we
use in the body of the paper.  In Section \ref{SecA.2} we recall
some standard combinatorics related to Weyl modules and their
interaction with certain tensor and restriction functors.  In
Section \ref{SecA.3} we prove some technical lemmas.

\label{AppendixB}

\subsection{Weyl filtrations}
\label{SecA.2}

\begin{theorem}
\label{Donkin}
\begin{enumerate}
\item Let $V_{n+1}$ be a Weyl module.  Then the $G_n$-module $R_i(V_{n+1})$ admits a Weyl filtration for all $i \in \Z$ (Proposition II.4.24,\cite{J}).
\item For any $V_n,W_{n} \in Rep(G_n)$ admitting Weyl filtrations, the $G_n$-module $V_n\otimes W_n$ also admits a Weyl filtration (Proposition II.4.21,\cite{J}).
\end{enumerate}
\end{theorem}

We will make use of the following classical result about branching
from $G_{n+1}$ to $G_{n}$ in characteristic zero.  (Recall that if
$p=0$ then $Rep(G_{n+1})$ is a semisimple category.)

\begin{lemma}\label{Multiplicity-Free}
Suppose $p=0$ and let $V \in Rep(G_{n+1})$ be a simple module.
\begin{enumerate}
\item  The $G_n$-module $R(V)$ is a multiplicity-free.  In particular, if $n \geq k$ and $\lambda \in \Lambda_k$ then
\begin{equation}
\label{Eq13} R_{1}(V_{n+1}(\lambda)) \simeq \bigoplus V_n(\mu)
\end{equation}
the sum over all $\mu \in \Lambda_{k-1}$ such that $ \xymatrix{
 \mu \ar[r] & \lambda}
$ (cf. Theorem 8.1.1, \cite{GW}).
\item
The $G_n$-module $\ff^n \otimes V$ is a multiplicity-free.
Precisely,
\begin{equation}
\label{Eq14} \ff^n \otimes V_n(\lambda) \simeq \bigoplus V_n(\mu)
\end{equation}
where the sum is over all $\mu$ such that $ \xymatrix{
 \lambda \ar[r] & \mu}
$ (cf. Corollary 9.2.4, \cite{GW}).
\end{enumerate}
\end{lemma}

Now we prove the analogue of this lemma for positive characteristic.

\begin{lemma}
\label{Weyl-Subquotient} Let $n \geq k$, $\lambda \in \Lambda_k$ and
consider the Weyl module $V_n(\lambda)$. \begin{enumerate}
\item
 Then $R_1(V_{n+1}(\lambda))$ has a Weyl filtration and the corresponding Weyl factors occur with multiplicity one and have precisely those highest weights that appear in the right side of (\ref{Eq13}).

\item Similarly, $\ff^n \otimes V_n(\lambda)$ has a Weyl filtration and the corresponding Weyl factors occur with multiplicity one and have precisely those highest weights that appear in the right side of (\ref{Eq14}).
\end{enumerate}
\end{lemma}
\begin{proof}
The modules under consideration have Weyl filtrations by Theorem
\ref{Donkin}.  Therefore, as elements of the integral group algebra
$\Z[X(T_n)]$, the characters $char_{\ff}(R_1(V_{n+1}(\lambda)))$ and
$char_{\ff}(V_n(\lambda)\otimes \ff^n)$ do not depend on the
characteristic of $\ff$.  In characteristic zero we know by the
above lemma that
$$
char_{\C}(R_1(V_{n+1}(\lambda)))=\sum char_{\C}(V_{n}(\mu))
$$
the sum over all $\mu$ such that $ \xymatrix{
 \mu \ar[r] & \lambda}$, and
 $$
 char_{\C}(V_n(\lambda)\otimes \C^n)=\sum char_{\C}(V_n(\mu))
 $$
the sum over all $\mu$ such that $ \xymatrix{
 \lambda \ar[r] & \mu}$.  Therefore, the same formulas hold with $\C$ replaced by $\ff$, and the lemma follows.
\end{proof}

\subsection{Some technical lemmas}
\label{SecA.3}

\begin{lemma}{\label{RepLem1}}
  Let $V_{n+1} \in Rep(G_{n+1})$.  Then the operator
  $x_{n,n+1}\cdot x_{n+1,n}$ acts on the space $R_1\circ R_0(V_{n+1})$ by $1$.
\end{lemma}

\begin{proof}
Let $i,j \in \mathbb{Z}$.  First we show that $x_{n,n+1}$ defines an
operator:
\begin{equation}
\label{EqEn} x_{n,n+1}:R_{i} R_j(V_{n+1}) \rightarrow R_{i+1} \circ
R_{j-1}(V_{n+1}).
\end{equation}
Indeed, for $k \in \{1,...,n+1\}$, let $\zeta_{k}:\ff^{\times}
\rightarrow G_{n+1}$ be the one-parameter subgroup
$$
z \mapsto diag(1,...,z,...1),
$$
where $z$ occurs in the $k^{th}$ position.  Now suppose $v \in R_{i}
\circ R_j(V_{n+1})$ and $z \in \ff^{\times}$.  Then,
\begin{align*}
\zeta_{n}(z) \cdot x_{n,n+1} \cdot v &= \zeta_{n}(z) \cdot x_{n,n+1}
\cdot \zeta_{n}(z^{-1}) \cdot \zeta_{n}(z)\cdot v \\ &=
z^{i+1}x_{n,n+1} \cdot v
\end{align*}
and similarly $\zeta_{n+1}(z) \cdot x_{n,n+1} \cdot
v=z^{j-1}x_{n+1,n} \cdot v$.  This proves (\ref{EqEn}).

By (\ref{EqEn}), $x_{n,n+1}:R_{1} \circ R_{0}(V_{n+1}) \rightarrow
R_{2} \circ R_{-1}(V_{n+1})$.  Since $V_{n+1}$ is a polynomial
representation, $R_{2}\circ R_{-1}(V_{n+1})=0$, and therefore
$x_{nn+1} \cdot v =0$ for all $v \in R_{1} \circ R_{0}(V_{n+1})$.
Therefore for $v \in R_{1} \circ R_{0}(V_{n+1})$,
\begin{align*}
x_{n,n+1}x_{n+1,n}.v = (x_{n,n}-x_{n+1,n+1}).v =v.
\end{align*}
\end{proof}

\begin{lemma}{\label{RepLem2}}
Let $V_{n+1}$ be a representation of $G_{n+1}$.  Then the following
identity of operators holds on $R_1 \circ R_1(V_{n+1})$:
\begin{equation*}
1-x_{n+1,n}x_{n,n+1}=s_n^{-1}.
\end{equation*}
\end{lemma}
\begin{proof}
By similar methods as applied in the proof of the above lemma, it
follows that $x_{n,n+1}^2 \cdot v =0$ for $v \in R_1 \circ
R_1(V_{n+1})$.  Moreover, such $v$ are of weight $(1,1)$ relative to
the torus of $GL_2$.  By the representation theory of $GL_2$, if
$char(\ff)>2$ then, all polynomial representation of $GL_2$ of
degree 2 is semisimple. Hence
$$
R_1 \circ R_1(V_{n+1}) \subset  I^{(1,1)} \oplus I^{(2,0)}
$$
where $I^{(i,j)}$ is the isotypic component of $V_{n+1}|_{G_{n-1}}$
corresponding to  the irreducible representation of $GL_2$ of
highest weight $(i,j)$.

For any  $v\in R_1 \circ R_1(V_{n+1})$,  we decompose v as
$v=v_{1,1}+v_{2,0}$, where $v_{(i,j)}\in I^{(i,j)}$. In particular
$v_{(2,0)}$ lies in the $(1,1)$-weight space of $I^2$.  By
elementary theory of $gl_2$, it is easy to check that for either
$v_{(1,1)} $ or $v_{(2,0)}$,
$$
(1-x_{n+1,n}x_{n,n+1})(v_{(i,j)})=s_n^{-1}(v_{(i,j)}).
$$
\end{proof}

\begin{lemma}
\label{Last Lemma} Let $g:V_{n+1}\to \ff^{n+1}\otimes I_0(W_n)$ be a
morphism of $G_{n+1}$-modules. If $v\in R_1(V_{n+1})$ and set
$g(v)=\sum_{i=1}^{n+1} x_i\otimes g_i$, then
\begin{equation}
\label{Eqtn6} x_{n+1,n+1}\cdot g_i= \left\{
\begin{array}{lr}
g_{i} \text{   if } i=1,...,n  &  \\
0 \text{  if } i =n+1
\end{array}
\right.
\end{equation}
\begin{equation}
\label{Eqtn7} x_{k,n+1}\cdot g_{n+1}=0 \text{  if } k =1,...,n
\end{equation}
\end{lemma}

\begin{proof}By hypothesis, $x_{n+1,n+1}\cdot v=v$, and since $g$ is a morphism of $G_{n+1}$-modules, $x_{n+1,n+1}\cdot g(v)=f(v)$.  Therefore,
$$
\sum_{i=1}^{n}x_i\otimes x_{n+1,n+1}\cdot g_i + x_{n+1}\otimes
g_{n+1} +x_{n+1}\otimes x_{n+1,n+1}\cdot
g_{n+1}=\sum_{i=1}^{n+1}x_i\otimes g_i,
$$
which implies formula (\ref{Eqtn6}).

To prove the second formula, recall that $G_1 \subset G_{n+1}$ acts
on $I_{0}(W_n) $ semi-simply. For an element $t\in G_1$, we have
$$
t\cdot(x_{k,n+1}\cdot g_{n+1})=tx_{k,n+1}t^{-1}\cdot t\cdot
g_{n+1}=t^{-1}(x_{k,n+1}\cdot g_{n+1}),
$$
and so $x_{k,n+1}\cdot g_{n+1}$ is of weight $-1$ for the action of
$G_1$. But $I_{0}(W_n)$ is a polynomial $G_{n+1}$-representation, so
in particular all the weights of $G_1$ on $I_{0}(W_n)$ are
nonnegative.  Therefore $x_{k,n+1}\cdot g_{n+1}$ has to be zero, for
$k=1,...,n$.
\end{proof}

\section{Polynomial functors}

\subsection{The functor category}
\label{thefunctorcategory}

Recall that $\ff$ is  algebraically closed.  The category of finite
dimensional vector spaces over $\ff$ is denoted $Vect_{\ff}$.  In
\cite{FS97}, Friedlander and Suslin introduce the category of
\emph{strict polynomial functors of finite degree}. Their category,
whose objects consists of certain endofunctors of $Vect_{\ff}$, will
be denoted by $\PP$.

  For $V,W \in Vect_{\ff}$, \emph{polynomial maps} from $V$ to $W$ are by definition elements of $S(V^{*})\otimes W$, where $S(V^*)$ denotes the symmetric algebra of the linear dual of $V$.  Elements of $S^{d}(V^{*})\otimes W$ are said to be of \emph{degree} $d$.
\begin{definition}
  The objects of the category $\PP$ are functors $T:Vect_{\ff} \rightarrow Vect_{\ff}$ that satisfy the following properties:
\begin{enumerate}
\item for any $V,W \in Vect_{\ff}$, the map of vector spaces
$$
Hom_{\ff}(V,W) \rightarrow Hom_{\ff}(T(V),T(W))
$$
is polynomial, and
\item the degree of the map
$$
End_{\ff}(V) \rightarrow End_{\ff}(T(V))
$$
is bounded uniformly for all $V \in Vect_{\ff}$.
\end{enumerate}
The morphisms in $\PP$ are natural transformations of functors.
\end{definition}

\subsection{The canonical equivalence}
\label{RemP}

We now show that the categories $\M$ and $\PP$ are canonically
equivalent.

Let $T \in \PP$.  By functoriality $T(\ff^n)$ carries an algebraic
action of $G_n$.  The representation $T(\ff^n)$  is polynomial
(Proposition 3.8, \cite{F01}). There exists a canonical functor
$\Phi:\PP \to \M$ defined as follows:
$$\Phi(T)=(T(\ff^n),\alpha_n)_{n=0}^{\infty},$$ where $\alpha_n:
R_0(T(\ff^{n+1}))\to T(\ff^n)$ is the map induced from the natural
$G_n$-equivariant projection $\pi_n:\ff^{n+1}\to \ff^n$,
\begin{equation*}
\xymatrix{ T(\ff^{n+1}) \ar[r]^{T(\pi_n)} &  T(\ff^n) \\
R_0(T(\ff^{n+1})) \ar@{^{(}->}[u] \ar[ur]_{\alpha_n} }
\end{equation*}
We need to show that $\alpha_n$ is an isomorphism for any $n$, and
thus $(T(\ff^n),\alpha_n)$ is a well-defined object in $\M$.

Let $\Gamma^k \in \M$ be the $k$-th divided power of vector spaces,
i.e. $\Gamma^k(V)=(\otimes^kV)^{S_k}$, and let $\Gamma^{k,n}$ be the
polynomial functor $\Gamma^k\circ Hom_\ff(\ff^n,\cdot)$.  Note that
the action of $G_n$ on $\ff^n$ induces an action on
$Hom_\ff(\ff^n,\cdot)$ and hence on $\Gamma^{k,n}$.

\begin{lemma}
The map $\alpha_n: R_0(T(\ff^{n+1}))\to T(\ff^n)$ is an isomorphism,
and thus the assignment $T \mapsto
(T(\ff^n),\alpha_n)_{n=0}^{\infty}$ defines a functor $\Phi:\PP \to
\M$.  \end{lemma}
\begin{proof}
We can assume $T$ is of degree $k$.  The $G_n$ action on the functor
$\Gamma^{k,n}$ induces a representation of $G_n$ on the vector
space $Hom_{\PP}(\Gamma^{k,n},T)$. By Theorem 2.10 in \cite{FS97},
$T(\ff^n)$ is canonically isomorphic to $Hom_{\PP}(\Gamma^{k,n},T)$
as $G_n$-modules.

Thus we need to check that $R_0(Hom_{\PP}(\Gamma^{k,n},T))\simeq
Hom_{\PP}(\Gamma^{k,n-1},T)$. The functor $\Gamma^{k,n}$ can be
decomposed canonically as
$$\Gamma^{k,n}=\bigoplus_{k_1+k_2\cdots +k_n=k} \Gamma^{k_1}\otimes\cdots \otimes \Gamma^{k_n}.$$
By Corollary 2.12 in \cite{FS97}, $\Gamma^{k_1}\otimes\cdots \otimes
\Gamma^{k_n}$ exactly represents the weight space of $T(\ff^n)$ with
weight $(k_1,k_2,...,k_n)$. In other words,
$$
Hom_{\PP}(\Gamma^{k_1}\otimes\cdots \otimes \Gamma^{k_n},T) \simeq
T(\ff^n)^{k_1,...,k_n}
$$
where $T(\ff^n)^{k_1,...,k_n}$ is the weight space corresponding to
the character $(k_1,...,k_n)$. Hence
\begin{eqnarray*}
R_0(Hom_{\PP}(\Gamma^{k,n},T)) &\simeq& \bigoplus_{k_1+k_2\cdots
+k_{n-1}=k} Hom_{\PP}(\Gamma^{k_1}\otimes\cdots \otimes
\Gamma^{k_{n-1}},T) \\ &\simeq& Hom_{\PP}(\Gamma^{k,n-1},T).
\end{eqnarray*}

\end{proof}

\begin{proposition}
The functor $\Phi: \PP \to \M$ is an equivalence.
\end{proposition}
\begin{proof}
Let $\PP(k) $ be the category of strict polynomial functors of
degree $k$. By Lemma 2.6 in \cite{FS97},
$\PP=\bigoplus_{d=0}^{\infty} \PP(k)$. Recall that
$\M=\bigoplus_{n=0}^{\infty} \M(k)$.  It is clear that $\Phi$
preserves the degree, i.e. $\Phi(\PP(k)) \subset  \M(k)$.

Let  $\Phi_n(k): \PP(k) \to \M_n(k)$ be the functor mapping $T$ to
$T(\ff^n)$.  Obviously $\Phi_n(k)=\Psi_n(k) \circ \Phi$, i.e. the
following diagram commutes,
$$
\xymatrix{ \PP(k) \ar[r]^\Phi \ar[d]_{\Phi_n(k)}&  \M(k)\ar[dl]^{\Psi_n(k)} \\
\M_n(k) }
$$
To prove $\Phi:\PP\to \M$ is an equivalence, it is enough to prove
that $\Phi:\PP(k) \to \M(k)$ is an equivalence for every $k\geq0$.
To show this, for each $k$, simply choose some $n\geq k$.  Then by
Lemma 3.4 in \cite{FS97} and Proposition \ref{Forgetful-Functor},
$\Phi_n(k)$ and $\Psi_{n}(k)$ are both equivalences. It follows that
$\Phi:\PP(k) \to \M(k)$ is also.
\end{proof}

\end{document}